\documentclass{article}

\usepackage[english]{babel}
\usepackage{amsmath,amsfonts,amssymb,amsthm}
\usepackage{amscd}
\usepackage{latexsym}
\usepackage{makeidx}

\input xy
\xyoption{all}

\usepackage{vmargin}
\setmarginsrb{22mm}{2mm}{25mm}{25mm}%
            {15mm}{15mm}{15mm}{15mm}

\numberwithin{equation}{section}

\newtheorem{em-deff}{Definition}[section]
\newtheorem{lem}[em-deff]{Lemma}
\newtheorem{theo}[em-deff]{Theorem}
\newtheorem{cor}[em-deff]{Corollary}

\newtheorem{prop}[em-deff]{Proposition}

\newtheorem{em-example}[em-deff]{Example}
\newtheorem{claim}[em-deff]{Claim}

\newtheorem{em-remark}[em-deff]{Remark}
\newtheorem{question}[em-deff]{Question}

\newenvironment{exmp}{\begin{em-example} \em }{ \end{em-example}}
\newenvironment{rem}{\begin{em-remark} \em }{ \end{em-remark}}
\newenvironment{defn}{\begin{em-deff} \em }{ \end{em-deff}}

\newcommand{\N}{\mathbb N}
\newcommand{\R}{\mathbb R}
\newcommand{\Z}{\mathbb Z}

\title{Arnautov's problems on semitopological isomorphisms}

\author{Dikran Dikranjan \and Anna Giordano Bruno}

\date{}

\begin{document}

\maketitle

\abstract{
Semitopological isomorphisms of topological groups were introduced by Arnautov \cite{Ar}, who posed several questions related to compositions of semitopological isomorphisms and the groups $G$ (we call them Arnautov groups) such that for every group topology $\tau$ on $G$ every semitopological isomorphism with domain $(G,\tau)$ is necessarily open (i.e., a topological isomorphism). We propose a different approach to these problems by introducing appropriate new notions, necessary for a deeper understanding of Arnautov groups. This allows us to find some partial answers and many examples. In particular, we discuss the relation with minimal groups and non-topologizable groups.
}

%

\section{Introduction}

It is easy to prove that for every continuous isomorphism $f:(G,\tau)\to(H,\sigma)$ of topological  groups, there exist a topological group $(\widetilde G,\widetilde\tau)$ containing $G$ as a topological subgroup and an open continuous homomorphism $\widetilde f:(\widetilde G,\widetilde\tau)\to (H,\sigma)$ extending $f$ \cite[Theorem 1]{Ar} (see also \cite[Theorem 1.1]{AGB} for continuous surjective homomorphisms).

The following notion is motivated by the fact that it is not always possible to prove the existence of such $\widetilde G$ and $\widetilde f$, asking $G$ to be also a {\em normal} subgroup of $\widetilde G$ (see also \cite{Ar0} for topological rings).

\begin{defn}{\cite[Definition 2]{Ar}}
A continuous isomorphism $f:(G,\tau)\to (H,\sigma)$ of topological groups is \emph{semitopological} if there exist a topological group $(\widetilde G,\widetilde\tau)$ containing $G$ as a topological normal subgroup and an open continuous homomorphism $\widetilde f:(\widetilde G,\widetilde\tau)\to (H,\sigma)$ extending $f$.
\end{defn}

In other words semitopological isomorphisms are restrictions of open continuous surjective homomorphisms to normal subgroups.
Obviously the class of semitopological isomorphisms contains the class of topological isomorphisms.

\medskip
Arnautov characterized semitopological isomorphisms \cite[Theorem 4]{Ar}. We give his characterization in terms of commutators and of thin subsets, as done in \cite{AGB}. 

For a neighborhood $U$ of the neutral element $e_G$ of a topological group $G$ call a subset $M$ of $G$ \emph{$U$-thin} if $\bigcap\{x^{-1}Ux:x\in M\}$ is still a neighborhood of $e_G$ (i.e., there exists a neighborhood $U_1$ of $e_G$ in $G$ such that $x U_1 x^{-1}\subseteq U$ for every $x\in M$). The subsets $M$ of $G$ that are $U$-thin for every $U$ are precisely the \emph{thin sets} in the sense of Tkachenko \cite{T,T1}. For example compact sets are thin.

\begin{theo}\label{semitop}\emph{\cite[Theorem 4]{Ar}}
Let $(G,\tau)$ and $(H,\sigma)$ be topological groups. Let $f:(G,\tau)\to(H,\sigma)$ be a continuous isomorphism. Then $f$ is semitopological if and only if for every $U\in\mathcal V _{(G,\tau)}(e_G)$:
\begin{itemize}
\item[(a)]there exists $V\in\mathcal V_{(H,\sigma)}(e_H)$ such that $f^{-1}(V)$ is $U$-thin;
\item[(b)]for every $g\in G$ there exists $V_g\in\mathcal V_{(H,\sigma)}(e_H)$ such that $[g,f^{-1}(V_g)]\subseteq U$.
\end{itemize}
\end{theo}

In \cite{AGB} we extended the notion of semitopological isomorphism introducing semitopological \emph{homomorphisms}. We defined new properties and considered particular cases in order to give internal conditions, similar to those of Theorem \ref{semitop}, which are sufficient or necessary for a continuous surjective homomorphism to be semitopological. Finally we established various stability properties of the class of all semitopological homomorphisms. Many particular cases are considered and they turn out to be useful also in this paper as well as other particular results; for those we will give references.

\bigskip
In Section \ref{properties} we give general properties of semitopological isomorphisms and see some stability properties of the class $\mathcal S_i$ of all semitopological isomorphisms. In fact it has been proved in \cite{Ar} that the class $\mathcal S_i$ is stable under taking subgroups, quotients and products, but not under taking compositions.

\bigskip
The aim of this paper is to discuss and answer the following problems raised by Arnautov \cite{Ar}:

\vspace{7pt}
\noindent\textbf{Problem A.} \cite[Problem 14]{Ar}
\emph{Find groups $G$ such that for every group topology $\tau$ on $G$ every semitopological isomorphism $f:(G,\tau)\to(H,\sigma)$, where $(H,\sigma)$ is a topological group, is open.}
\vspace{7pt}

\vspace{7pt}
\noindent\textbf{Problem B.} \cite[Problem 13]{Ar}
Let $\mathcal K$ be a class of topological groups. Find $(G,\tau)\in\mathcal K$ such that every semitopological isomorphism $f:(G,\tau)\to(H,\sigma)$ in $\mathcal K$ is open.
\vspace{7pt}

The third problem concerns compositions:

\vspace{7pt}
\noindent\textbf{Problem C.} \cite[Problem 15]{Ar}
\emph{\begin{itemize}
  \item[(a)] Which are the continuous isomorphisms of topological groups that are compositions of semitopological isomorphisms?
  \item[(b)] Is every continuous isomorphism of topological groups composition of semitopological isomorphisms?
\end{itemize}}
\vspace{7pt}

\subsection{The Open Mapping Theorem and its weaker versions}\label{OMT}

According to the Banach's open mapping theorem every surjective continuous
linear map between Banach spaces is  open
{\cite{Ba}. As a generalization, Pt\' ak \cite{Pt} introduced the notion
of $B$-{\it completeness}
for the class of linear topological spaces. It was based on the property  weaker than openness,
that can be
formulated also in the larger class of topological groups as follows: a
homomorphism $f:G\to H$ of topological groups is called {\it almost open},
if
for every neighborhood $U$ of $e_G$ in $G$ the image $f(U)$ is dense in some
neighborhood of $e_H$ in $H$.
A topological group $G$ is  {\it $B$-complete} (respectively, {\it
$B_r$-complete})  if every continuous almost open surjective homorphism
(respectively, isomorphism) from $G$ to any Hausdorff group is open.
These groups were intensively studied in the sixties and the seventies
(\cite{Br}, \cite{Gr}, \cite{Hu}, \cite{Su}).  It was shown by Husain
\cite{Hu} that  locally compact groups as well as complete metrizable
groups are $B$-complete. Brown \cite{Br} found a common generalization of
this fact by proving that \v Cech-complete groups are  $B$-complete.

The following notion introduced by Choquet (see Do\" \i tchinov \cite{Do})
and Ste\-phen\-son \cite{St} in 1970 takes us closer to the spirit of Banach
open mapping theorem:

\begin{defn}
A Hausdorff group topology $\tau$ on a group $G$ is \emph{minimal} if for
every continuous isomorphism $f:(G,\tau)\to H$, where $H$ is a
Hausdorff topological group, $f$ is a topological isomorphism. Call $G$ \emph{totally
minimal} if for every continuous homomorphism $f:(G,\tau)\to H$, where $H$
is
Hausdorff, $f$ is open.
\end{defn}

Clearly, the totally minimal groups are precisely the topological groups that satisfy the Banach's open mapping theorem. Since all surjective homomorphisms between precompact groups are almost open, a precompact group is $B_r$-complete (respectively, $B$-complete) if and only if it is minimal (respectively, totally minimal). In particular, the $B_r$-complete precompact abelian groups coincide with
the minimal abelian groups as every minimal abelian group is precompact according to the celebrated Prodanov-Stoyanov's theorem.
According to this theorem, an infinite minimal abelian group is never discrete. This radically changes in the non-abelian case. In the forties Markov asked whether every infinite group $G$ is topologizable (i.e.,
admits a non-discrete Hausdorff group topology).

\begin{defn}
A group $G$ is:
\begin{itemize}
       \item \emph{Markov} if the discrete topology $\delta_G$ is the unique Hausdorff group topology
on $G$ (i.e., $\delta_G$ is minimal);
       \item \emph{totally Markov} if $G/N$ is a Markov for every
$N\triangleleft G$.
\end{itemize}
\end{defn}

Obviously totally Markov implies Markov and finite groups are totally Mar\-kov, while every simple Markov group is totally Markov. Denote by ${\mathfrak M}$ and ${\mathfrak M}_t$ the classes of all Markov and totally Markov groups respectively. Markov's question (on whether ${\mathfrak M}$ contains infinite groups), was answered only thirty-five years later by Shelah \cite{Sh} (who needed CH for
his example, resolving simultaneously also Kurosh' problem) and Ol$'$shanskii (who made use  of the properties of remarkable Adian's groups) \cite{Ol}.

A smaller class of groups arose in the solution of a specific problem related to categorical compactness in \cite{DU}: namely the subclass of ${\mathfrak M}_t$ consisting of those groups $G\in {\mathfrak M}_t$ such that every subgroup of $G$ belongs to ${\mathfrak M}_t$ as well (these groups were named \emph{hereditarily non-topologizable} by Luk\'acs \cite{Lu}). It is still an open question whether an infinite hereditarily non-topologizable group exists (\cite{DS,DU,Lu}).

A possibility to relax the strong requirement in the open mapping theorem in the definition of minimal groups is to restrict the class of topological groups:

\begin{defn}\label{K-min}
Let $\mathcal K$ be a class of topological groups. A topological group $(G,\tau)\in\mathcal K$ is \emph{$\mathcal K$-minimal} if $(G,\sigma)\in\mathcal K$ and $\sigma\leq\tau$ imply $\tau=\sigma$.
\end{defn}

When $\mathcal K$ is the class of all metrizable abelian groups, $\mathcal K$-minimal groups are precisely the minimal abelian groups that are metrizable \cite{DPS}, but in general a $\mathcal K$-minimal group need not be minimal. Anyway, if $\mathcal H$ is the class of all Hausdorff topological groups, then $\mathcal H$-minimality is precisely the usual minimality.

Recently new generalizations of minimality for topological groups  were considered (relative minimality and co-minimality, cf. \cite{DM,Sl}).

\subsection{Main Results}

The next definition reminds the $B_r$-completeness (since we impose openness only on certain continuous isomorphisms, namely, the
semitopological ones):

\begin{defn}\label{Def_A-min}  
A group topology $\tau$ on $G$ is \emph{A-complete} if for every group topology $\sigma$ on $G$, $\sigma\leq\tau$ and $id_G:(G,\tau)\to(G,\sigma)$ semitopological imply $\tau=\sigma$.
\end{defn}

Finally, we can formulate the notion that captures the core of Problem A:

\begin{defn}\label{Def_A-group}
A group $G$ is an \emph{Arnautov group} if every group topology on $G$ is
A-complete (i.e., if for every pair of group topologies $\tau,\sigma$ on $G$ with
$\sigma<\tau$, $id_G:(G,\tau)\to(G,\sigma)$ is not semitopological).
\end{defn}

Hence Problem A can be formulate also as follows:
\emph{characterize the groups $G$ such that every group topology on $G$ is A-complete, that is, characterize the Arnautov groups.}

\smallskip
We denote by $\mathfrak A$ the class of all Arnautov groups.

\medskip
Ta\u\i manov \cite{Tai} introduced the group topology $T_G$ on a group $G$, which has the family of the centralizers of the elements of $G$ as a prebase of the filter of the neighborhood of $e_G$. This topology was introduced with the aim of the topologization of abstract groups with Hausdorff group topologies.

Since $id_G:(G,\delta_G)\to (G,\sigma)$ is semitopological if and only if $\sigma\geq T_G$ (see \cite[Corollary 5.3]{AGB} or Remark \ref{discrhom}) and we are studying Arnautov groups, we need to impose that $T_G$ is discrete and we introduce the following notion.

\begin{defn}\label{tai-def}
A group $G$ is:
\begin{itemize}
	\item \emph{Ta\u\i manov} if $T_G=\delta_G$;
	\item \emph{totally Ta\u\i manov} if $G/N$ is Ta\u\i manov for every $N\triangleleft G$.
\end{itemize}
\end{defn}

Obviously every simple Ta\u\i manov group is totally Ta\u\i manov.

We denote by $\mathfrak T$ and $\mathfrak T_t$ the classes of Ta\u\i manov and totally Ta\u\i manov groups respectively.

\medskip
Since Problem A in its full generality seems to be hard to handle (because of two universal quantifiers), we start considering a particular case, that is when the discrete topology on a group $G$ is A-complete and we prove that 
for a group $G$ the discrete topology is A-complete if and only if $G\in\mathfrak T$
(see Theorem \ref{discrAminz1}). Moreover we extend this result for almost trivial topologies (which are obtained from the trivial ones by extension, as their name suggests --- see Section \ref{at-sec}), characterizing in Theorem \ref{almdiscr-solution} when an almost trivial topology is A-complete in terms of $\mathfrak T$.

Moreover $\mathfrak T_t$ contains $\mathfrak A$, but we don't know if they coincide (see Theorem \ref{perfect-taimanov} and Question \ref{perfect-taimanov_question}).

Example \ref{S(Z)} considers properties of $S(\Z)$ related to Problem A. First of all it shows that A-completeness has a behavior different from that of the usual minimality. Indeed we see that $S(\Z)$ admits at least two different but comparable A-complete group topologies. Moreover $S(\Z)$ is not Ta\u\i manov and consequently not Arnautov. Nevertheless $S(\Z)/S_\omega(\Z)$ is totally Ta\u\i manov but we do not know if it is also Arnautov (see Question \ref{S(Z)/S}).

\medskip
This question can be seen as a first step in answering the following one, which could give an infinite example of a simple infinite Markov group without assuming CH (see Question \ref{S(Z)/S-markov}): 
\begin{center}
does $S(\Z)/S_\omega(\Z)\in \mathfrak M$?
\end{center}

But the situation can be reversed: if $S(\Z)/S_\omega(\Z)\in \mathfrak M$ then $S(\Z)/S_\omega(\Z)\in \mathfrak A$, in view of Corollary \ref{finite-arnautov}(b), which says that every simple Markov group is necessarily Arnautov.
Thanks to this property we have the unique infinite Arnautov group that we know at the moment, that is Shelah group, which is an infinite simple Markov group constructed under CH \cite{Sh} (see Example \ref{infinite-arnautov}).

\bigskip
The next definition, combining Definition \ref{Def_A-min} (A-completeness) and Definition \ref{K-min} ($\mathcal K$-minimality) will allow us to handle easier Problem B.

\begin{defn}
For a class $\mathcal K$ of topological groups, a topological group $(G,\tau)$ from $\mathcal K$ is \emph{A$_\mathcal K$-complete} if $(G,\sigma)\in\mathcal K$, $\sigma\leq\tau$ and $id_G:(G,\tau)\to(G,\sigma)$ semitopological imply $\tau=\sigma$.
\end{defn}

Let $\mathcal G$ be the class of all topological groups.

\begin{rem}\label{A_K-compl-rem}
\begin{itemize}
	\item[(a)]Obviously $\mathcal K$-minimality implies A$_\mathcal K$-completeness and $\mathcal K$-minimality coincides with A$_\mathcal K$-completeness whenever all groups in $\mathcal K\subseteq \mathcal G$ are abelian.
	\item[(b)]Moreover A-completeness coincides with A$_\mathcal G$-completeness. So Problem A can be seen as a particular case of Problem B, namely with $\mathcal K=\mathcal G$.
	\item[(c)]If $\mathcal K\subseteq\mathcal K'$ are classes of topological groups, then for every $G\in\mathcal K$ A$_{\mathcal K'}$-complete implies A$_\mathcal K$-complete. In particular, if $\mathcal K\subseteq\mathcal G$ and $G\in\mathcal K$, then $G$ A-complete implies $G$ A$_\mathcal K$-complete.
\end{itemize}
\end{rem}

Clearly A$_\mathcal H$-completeness is a generalization of minimality, since $\mathcal H$-mi\-ni\-ma\-li\-ty is precisely the usual minimality, which is intensively studied, as noted in Section \ref{OMT}. This is a strict generalization as shown by Example \ref{S(Z)-min}.

A topological group $G$ \emph{has small invariant neighborhoods} (i.e., \emph{$G$ is SIN}) if $G$ is thin (i.e., it has a local base at $e_G$ of neighborhoods invariant under conjugation). We prove that a topological group, which is SIN and A$_\mathcal H$-complete, is A-complete if and only it has trivial center (see Remark \ref{rem}). In particular, if $G$ is a group with trivial center, its discrete topology is A$_\mathcal H$-complete if and only if $G\in \mathfrak T$ (see Corollary \ref{A_H-min}). So also in this case Ta\u\i manov groups play a central role.

Moreover we give an example of a small class $\mathcal K$ in which each element is A$_\mathcal K$-complete (see Example \ref{K_R^0}). This class is built on the Heisenberg group $$\mathbb H_\mathbb R:=\begin{pmatrix}
    1 & \mathbb R & \mathbb R \\
    0 & 1 & \mathbb R\\
    0 & 0 & 1
\end{pmatrix},$$ that is the group of upper unitriangular $3\times 3$ matrices over $\mathbb R$, endowed with different group topologies.  The group $\mathbb H_\mathbb R$ is nilpotent of class $2$.

In a forthcoming paper \cite{DGB_x} we extend this example for generalized Heiseberg groups, that is, the group of upper unitriangular $3\times 3$ matrices over a unitary ring $A$.

\bigskip
In Example \ref{C(b)} we resolve negatively item (b) of Problem C.
Moreover Theorem \ref{n-step_semitop} answers partially (a), in the case when the topologies on the domain and on the codomain are the discrete and the indiscrete one respectively. Since we consider the trivial topologies, the condition that we find is exclusively algebraic. Indeed we prove that $id_G:(G,\delta_G)\to (G,\iota_G)$ is composition of $n$ semitopological isomorphisms if and only if $G$ is nilpotent of class $\leq n$, where $n\in\N_+$.

\subsection*{Notation and terminology}

We denote by $\mathbb R$, $\mathbb Q$, $\mathbb Z$, $\mathbb P$, $\mathbb N$ and $\mathbb N_+$ respectively the field of real numbers, the field of rational numbers, the ring of integers, the set of primes, the set of natural numbers and the set of positive integers.

Let $G$ be a group and $x,y\in G$. We denote by $[x,y]$ the commutator of $x$ and $y$ in $G$, that is $[x,y]=x y x^{-1}y^{-1}$ and for $x\in G$ and a subset $Y$ of $G$ let $[x,Y]=\{[x,y]:y\in Y\}$.
More in general, if $H$ and $K$ are subgroups of $G$, let $$[H,K]=\langle [h,k]:h\in H, k\in K\rangle,$$ and in particular the derived $G'$ of $G$ is $G'=[G,G]$, that is, the subgroup of $G$ generated by all commutators of elements of $G$. The center of $G$ is $Z(G)=\{x\in G:x g=g x, \forall g\in G\}$ and for $g\in G$ the centralizer of $g$ in $G$ is $c_G(g)=\{x\in G: x g=g x\}$.

The diagonal map $\Delta:G\to G\times G$ is defined by $\Delta(g)=(g,g)$ for every $g\in G$. If $H$ is another group, we denote by $p_1:G\times H\to G$ and $p_2:G\times H\to H$ the canonical projections on the first and the second component respectively. If $f:G\to H$ is a homomorphism, denote by $\Gamma_f$ the graph of $f$, that is the subgroup $\Gamma_f=\{(g,f(g)):g\in G\}$ of $G\times H$.

If $\tau$ is a group topology on $G$ then denote by $\mathcal V_{(G,\tau)}(e_G)$ the filter of all neighborhoods of $e_G$ in $(G,\tau)$ and by $\mathcal B_\tau$ a base of $\mathcal V_{(G,\tau)}(e_G)$. If $X$ is a subset of $G$, $\overline X ^\tau$ stands for the closure of $X$ in $(G,\tau)$.

If $N$ is a normal subgroup of $G$ and $\pi:G\to G/N$ is the canonical projection, then $\tau_q$ is the quotient topology of $\tau$ in $G/N$. Moreover $N_\tau$ denotes the subgroup $\overline{\{e_G\}}^\tau$. The discrete topology on $G$ is $\delta_G$ and the indiscrete topology on $G$ is $\iota_G$.

For undefined terms see \cite{E,Fuchs}.

\section{Properties of semitopological isomorphisms}\label{properties}

In the next remark we discuss the possibility to consider only the case of one group $G$ endowed with two different topologies $\tau\geq\sigma$ taking $id_G:(G,\tau)\to(G,\sigma)$ as the continuous isomorphism:

\begin{rem}
Let $(G,\tau)$, $(H,\sigma)$ be topological groups and $f:(G,\tau)\to (H,\eta)$ a continuous isomorphism. Consider the topology $\sigma=f^{-1}(\eta)$ on $G$. This topology $\sigma$ is coarser than $\tau$ and so $id_G:(G,\tau)\to (G,\sigma)$ is a continuous isomorphism and $(G,\sigma)$ is topologically isomorphic to $(H,\eta)$. In particular 
\begin{quote}
$id_G:(G,\tau)\to(G,\sigma)$ is semitopological if and only if\\
$f:(G,\tau)\to(H,\eta)$ is semitopological.
\end{quote}
\end{rem}

Moreover the next proposition shows that semitopological is a ``local'' property, like the stronger property open. The proof is a simple application of Theorem \ref{semitop}.

\begin{prop}
Let $G$ be a group and $\tau,\sigma$ group topologies on $G$ such that $\sigma\leq\tau$. Then $id_G:(G,\tau)\to(G,\sigma)$ is semitopological if there exists a $\tau$-open subgroup $N$ of $G$ such that $id_G \restriction_N:(N,\tau \restriction_N)\to(N,\sigma \restriction_N)$ is semitopological.
\end{prop}

The following theorems show the stability for subgroups, quotients and products.

\begin{theo}\emph{\cite[Theorems 7 and 8]{Ar}}\label{subgroupst}\label{quotientst}
Let $G$ be a group, $\sigma\leq\tau$ group topologies on $G$ and suppose that $id_G:(G,\tau)\to(G,\sigma)$ is semitopological.
\begin{itemize}
	\item[(a)]If $A$ is a subgroup of $G$, then $id_A:(A,\tau \restriction_A)\to (A,\sigma \restriction_A)$ is semitopological.
	\item[(b)]If $A$ is a normal subgroup of $G$, then $id_{G/A}:(G/A,\tau_q)\to (G/A,\sigma_q)$ is semitopological.
\end{itemize}
\end{theo}

\begin{theo}\emph{\cite[Theorem 9]{Ar}, \cite[Theorem 6.15]{AGB}}\label{productst}
Let $\{G_i:i\in I\}$ be a family of groups and $\{\tau_i:i\in I\}$, $\{\sigma_i:i\in I\}$ families of group topologies such that $\sigma_i\leq\tau_i$ are group topologies on $G_i$ for every $i\in I$. Then $id_{G_i}:(G_i,\tau_i)\to(G_i,\sigma_i)$ is semitopological for every $i\in I$ if and only if $\prod_{i\in I}id_{G_i}:\prod_{i\in I}(G_i,\tau_i)\to\prod_{i\in I}(G_i,\sigma_i)$ is semitopological.
\end{theo}

The next lemma shows a cancellability property of compositions of semitopological isomorphisms.

\begin{lem}\label{primo_ok}\emph{\cite[Theorem 6.11]{AGB}}
Let $\sigma\leq\tau$ be group topologies on a group $G$. If $id_G:(G,\tau)\to(G,\sigma)$ is semitopological, then for a group topology $\rho$ on $G$ such that $\sigma\leq\rho\leq\tau$, $id_G:(G,\tau)\to(G,\rho)$ is semitopological.
\end{lem}

In a particular case, that is for initial topologies, the converse implication of Theorem \ref{quotientst}(b) holds true:

\begin{lem}\label{initial}
Let $G$ be a group and $N$ a normal subgroup of $G$. Let $\sigma\leq\tau$ be group topologies on $G/N$ and $\sigma_i\leq\tau_i$  the respective initial topologies on $G$. Then $id_{G}:(G,\tau_i)\to (G,\sigma_i)$ is semitopological if and only if $id_{G/N}:(G/N,\tau)\to (G/N,\sigma)$ is semitopological.
\end{lem}

In the next theorem we consider the particular cases when one of the two topologies on $G$ is trivial:

\begin{theo}\emph{\cite[Corollary 5]{Ar}, \cite[Corollary 5.11]{AGB}}\label{discrete}\label{ind}
Let $G$ be a group and $\tau$ a group topology on $G$. Then:
\begin{itemize}
	\item[(a)] $id_G:(G,\delta_G)\to(G,\tau)$ is semitopological if and only if $c_G(g)$ is $\tau$-open for every $g\in G$;
	\item[(b)] $id_G:(G,\tau)\to(G,\iota_G)$ is semitopological if and only if $G'\leq N_\tau$.
\end{itemize}
\end{theo}

Since $Z(G)\subseteq c_G(g)$ for every $g\in G$, by (a) $id_G:(G,\delta_G)\to(G,\tau)$ semitopological implies $Z(G)$ $\tau$-open.

The condition $G'\leq N_\tau$ in (b) is equivalent to say that $G'$ is indiscrete endowed with the topology inherited from $(G,\tau)$. Moreover, as noted in \cite{AGB}, it implies that $(G,\tau)$ is SIN.

\smallskip
For SIN groups condition (a) of Theorem \ref{semitop} is always verified, since SIN groups are thin, so only condition (b) remains:

\begin{prop}\label{SINsemitop}
Let $G$ be a group and $\sigma\leq \tau$ group topologies on $G$. Suppose that $(G,\tau)$ is SIN. Then $id_G:(G,\tau)\to(G,\sigma)$ is semitopological if and only if for every $U\in\mathcal V_{(G,\tau)}(e_G)$ and for every $g\in G$ there exists $V_g\in\mathcal V_{(G,\sigma)}(e_G)$ such that $[g,V_g]\subseteq U$.
\end{prop}

The next lemma gives a simple necessary condition of algebraic nature for a continuous isomorphism to be semitopological.

\begin{lem}\label{G,N<N}
Let $G$ be a group and $\sigma\leq \tau$ group topologies on $G$, such that $id_G:(G,\tau)\to(G,\sigma)$ is semitopological. Then $[G,N_\sigma]\leq N_\tau$.
\end{lem}
\begin{proof}
By Theorem \ref{semitop}, for every $U\in\mathcal V_{(G,\tau)}(e_G)$ and every $g\in G$, there exists $V_g\in\mathcal V_{(G,\tau)}(e_G)$ such that $[g,V_g]\subseteq U$. Consequently $[g,N_\sigma]\subseteq U$ for every $g\in G$, so $[g,N_\sigma]\subseteq N_\tau$ for every $g\in G$ and hence $[G,N_\sigma]\leq N_\tau$.
\end{proof}

\begin{cor}\label{hausdorff}
Let $G$ be a group and $\tau$ a group topology on $G$. If $\tau$ is Hausdorff, then $id_G:(G,\tau) \to (G,\iota_G)$ is semitopological if and only if $G$ is abelian.
\end{cor}
\begin{proof}
If $id_G:(G,\tau) \to (G,\iota_G)$ is semitopological, by Lemma \ref{G,N<N} $G'\leq N_\tau=\{e_G\}$ and hence $G$ is abelian. If $G$ is abelian every continuous isomorphism is semitopological.
\end{proof}

In particular $id_G:(G,\delta_G) \to (G,\iota_G)$ is semitopological if and only if the group $G$ is abelian.

\begin{prop}\label{t2,z=1,t2}
Let $G$ be a group and $\sigma\leq \tau$ group topologies on $G$, such that $id_G:(G,\tau)\to(H,\sigma)$ is semitopological. If $Z(G)=\{e_G\}$ and $\tau$ is Hausdorff, then $\sigma$ is Hausdorff as well.
\end{prop}
\begin{proof}
Since $N_\tau=\{e_G\}$ and $[G,N_\sigma]\leq N_\tau$ by Lemma \ref{G,N<N}, using the hypothesis $Z(G)=\{e_G\}$ we conclude that $N_\sigma=\{e_G\}$.
\end{proof}

\section{Almost trivial topologies}\label{at-sec}

In this section we introduce a class of group topologies containing the trivial ones and with nice stability properties; moreover we extend Theorem \ref{discrete} to this class.

\begin{defn}\cite[Definition 5.13]{AGB}\label{at-def}
A topological group $(G,\tau)$ is \emph{almost trivial} if $N_\tau$ is open in $(G,\tau)$.
\end{defn}

Since in this case $\tau$ is completely determined by the normal subgroup $N:=N_\tau$ of $G$, we denote an almost trivial topology on $G$ by $\zeta_N$, underling the role of the normal subgroup.

\smallskip
Every group topology on a finite group is almost trivial and every almost trivial group is SIN.

For example, for a group $G$, the discrete and the indiscrete topologies (i.e., the so-called trivial topologies) are almost trivial, with $\delta_G=\zeta_{\{e_G\}}$ and $\iota_G=\zeta_G$. This justifies the term used in Definition \ref{at-def}.

\begin{lem}\label{simplenonabrem}
Let $G$ be a simple non-abelian group and let $\tau$ be a group topology on $G$. Then either $N_\tau=G$ or $N_\tau=\{e_G\}$, that is, either $\tau=\iota_G$ or $\tau$ is Hausdorff, respectively. If $\tau$ is almost trivial, then $\tau$ is either discrete or indiscrete.
\end{lem}

The almost trivial topologies help also to express in simple terms topological properties:

\begin{rem}
Given a topological group $(G,\tau)$ and a normal subgroup $N$ of $G$, it is possible to consider the group topology obtained ``adding'' to the open neighborhoods also $N$ (since it is normal, it suffices to add $N$ to the prebase of the neighborhoods and all the intersections $U \cap N$, with $U \in \mathcal V_{(G,\tau)}(e_G)$, give the neighborhoods of $e_G$ in the new topology). This new topology is $\sup\{\tau,\zeta_N\}$.

\smallskip
For example, if $G$ is a group and $\tau$ its profinite topology, with $\mathcal B_\tau=\{N_\alpha\}_\alpha$, where the $N_\alpha$ are all the normal subgroups of $G$ of finite index, then $\tau=\sup_\alpha\zeta_{N_\alpha}$.
More in general, if $\tau$ is a linear topology on $G$, that is $\mathcal B_\tau=\{N_\alpha\}_\alpha$, where $N_\alpha$ are normal subgroups of $G$, then $\tau=\sup_\alpha\zeta_{N_\alpha}$.
\end{rem}

If $(G,\tau)$ is a topological group, let $\bar{\tau}$ denote the quotient topology of $(G,\tau)$ with respect to the normal subgroup $N_\tau$, which is indiscrete. Then $\bar{\tau}$ is Hausdorff.
Moreover $(G,\tau)$ is almost trivial if and only if $(G/N_\tau,\bar\tau)$ is discrete.

\medskip
Analogously it is possible to consider the case when a topological group $(G,\tau)$ has a discrete normal subgroup $D$ such that $(G/D,\tau_q)$ is indiscrete. For groups with this property we have a strong consequence:

\begin{lem}\label{case1}
Let $(G,\tau)$ be a topological group such that $D$ is a discrete normal subgroup of $(G,\tau)$ and $(G/D,\tau_q)$ is indiscrete. Then $(G,\tau)\cong D\times N_\tau$, where $D$ is discrete and $N_\tau$ is indiscrete. In particular $\tau$ is almost trivial.
\end{lem}
\begin{proof}
Pick a symmetric neighborhood $W$ of $e_G$ in $G$ such that $W^3\cap D = \{e_G\}$. 
Since $(G/D,\tau_q)$ is indiscrete, $D$ is dense in $G$, so $G=D W$. Let $w_1,w_2 \in W$. Then there exists $d\in D$ such that $w_1 w_2 \in d W$. 
Let $w_1 w_2=d w$ for some $w\in W$. Then $d=w_1 w_2 w^{-1}\in W^3\cap D = \{e_G\}$. So $w_1 w_2=w\in W$. Since $W$ is symmetric, this proves that $W$ is an open subgroup of $M$ with $W\cap D=\{e_G\}$. Hence the restriction of the canonical projection $G \to G/D$ to $W$ gives a topological isomorphism $W \cong (G/D,\tau_q)$. This shows that $W$ is an indiscrete group. Since $N_\tau\leq W$ is closed, we deduce that $W=N_\tau$. This proves that $N_\tau$ is open in $\tau$ and that $(G,\tau)\cong D\times N_\tau$.
\end{proof}

\subsection{Permanence properties of the almost trivial topologies}

The assignment $N\mapsto \zeta_N$ defines an order reversing bijection between the complete lattice $\mathcal N(G)$ of all normal subgroups of a group $G$ and the complete lattice $\mathcal{AT}(G)$ of all almost trivial group topologies on $G$. Let us note that the complete lattice $\mathcal{AT}(G)$ is not a sublattice of the complete lattice ${\mathcal T}(G)$ of {\em all} group topologies on $G$. Indeed, the meet of a family $\{\zeta_{N_i}:i\in I\}$ in $\mathcal{AT}(G)$ is simply $\zeta_{\bigcap_{i\in I} N_i}$, whereas the meet of a family
$\{\zeta_{N_i}:i\in I\}$ in ${\mathcal T}(G)$ is the group topology having as prebase of the neighborhoods at $e_G$
the family $\{\zeta_{N_i}:i\in I\}$ (in other words, the latter topology may be strictly weaker than the former one in case $I$ is infinite). 

The next lemma shows, among others, that the class of almost trivial groups is closed under taking subgroups and quotients.

\begin{lem}\label{almdiscrPP}
Let $(G,\zeta_N)$ be an almost trivial group, where $N$ is a normal subgroup of $G$. 
\begin{itemize}
  \item[(a)] For every subgroup $H$ of $G$:
	\begin{itemize}
  	\item[(a$_1$)] the topology induced on $H$ by $\zeta_N$ is almost trivial and coincides with $\zeta_{H\cap N}$; 
  	\item[(a$_2$)] the following conditions are equivalent: (i) $H$ is $\zeta_N$-open; 
   (ii) $H$ is $ \zeta_N$-closed; (iii) $H \geq N$.
	\end{itemize}
  \item[(b)] For every normal subgroup $N_0$ of $G$ the quotient topology of $\zeta_N$ on $G/N_0$ is almost trivial and coincides with $\zeta_{N_0N/N_0}$. 
\end{itemize}
\end{lem}

\begin{rem}\label{almdiscrSG}
In connection to item (a$_1$) of the previous lemma notice that if $H$ is an open subgroup of a topological group $G$ and $H$ is almost trivial, then also $G$ is almost trivial.
\end{rem}

Now we show that the class of almost trivial groups is stable also with respect to taking finite products.

\begin{lem}
Let $G_1$, $G_2$ be groups and $N_1$, $N_2$ normal subgroups of $G_1$, $G_2$ respectively. Then $\zeta_{N_1}\times \zeta_{N_2}=\zeta_{N_1\times N_2}$ on $G_1\times G_2$. 
\end{lem}

The next lemma follows directly from the definitions. 

\begin{lem}\label{almdiscrIndQ}
Let $G$ be a topological group and $N$ an indiscrete normal subgroup of $G$ such that $G/N$ is almost trivial. Then $G$ is almost trivial. 
\end{lem}

We want to generalize this lemma and we need the following concept.

\begin{defn}
For a class of topological groups ${\mathcal P}$ one says that ${\mathcal P}$ has {\em the three space property}, if a topological group $G$ belongs to ${\mathcal P}$ whenever $N\in {\mathcal P}$ and $G/N\in {\mathcal P}$ for some normal subgroup $N$ of $G$.
\end{defn}

For example the class of all discrete groups and the class of all indiscrete groups have the three space property. So the next result shows that the class of all almost trivial groups is the smaller class with the three space property containing all discrete and all indiscrete groups.

\begin{prop}\label{3space-at}
The class of almost trivial groups has the three space property.
\end{prop}
\begin{proof}
We have to prove that, in case $G$ is a group and $N$ a normal subgroup of $G$, if $\tau$ is a group topology on $G$ such that $(N,\tau\restriction_N)$ and $(G/N,\tau_q)$ are almost trivial, then $(G,\tau)$ is almost trivial.

Let $M$ be the normal subgroup of $G$ containing $N$ such that $M/N=N_{{\tau}_q}$. Then $M/N$ is indiscrete
and open in $G/N$. Consequently, $M$ is open in $(G,\tau)$. To end the proof we need to verify that $M$ is almost trivial (see Remark \ref{almdiscrSG}).

If $\tau\restriction_N$ is Hausdorff, equivalently it is discrete, since it is almost trivial, and by Lemma \ref{case1} $M$ is almost trivial.
So we consider now the general case. The subgroup $N_1:=N_\tau\cap N$ is the closure of $\{e_G\}$ in $M$. Then $N_1$ is a normal subgroup of $N$. Now the normal subgroup $N/N_1$ of the Hausdorff quotient group $M/N_1$ is almost trivial and consequently discrete. Moreover, the quotient $(M/N_1)/(N/N_1)\cong M/N$ is indiscrete. So by the previous case the group $M/N_1$ is almost trivial. Since the group $N_1$ is indiscrete, we can conclude with Lemma \ref{almdiscrIndQ}. 
\end{proof}

\subsection{Semitopological isomorphisms between almost trivial topologies}

Since every almost trivial group is SIN, it is possible to apply Proposition \ref{SINsemitop} instead of Theorem \ref{semitop} to verify if a continuous isomorphism is semitopological. In case the topology on the domain or that on the codomain is almost trivial, the conditions of Theorem \ref{semitop} become simpler:

\begin{prop}\label{sigma-at}
Let $(G,\sigma)$ be a topological group and let $\sigma\leq\tau$ be group topologies on $G$.
\begin{itemize}
	\item[(a)]If $\tau$ is almost trivial, then $id_G:(G,\tau)\to(G,\sigma)$ is semitopological if and only if for every $g\in G$ there exists $V_g\in\mathcal V_{(G,\sigma)}(e_G)$ such that $[g,V_g]\subseteq N_\tau$.
	\item[(b)]If $\sigma$ is almost trivial, then $id_G:(G,\tau) \to (G,\sigma)$ is semitopological if and only if $N_\sigma$ is $U$-thin for every $U\in\mathcal V_{(G,\tau)}(e_G)$ and $[G,N_\sigma]\leq N_\tau$.
\end{itemize}
\end{prop}
\begin{proof}
(a) follows from Proposition \ref{SINsemitop}.

(b) The necessity of the condition that $N_\sigma$ is $U$-thin for every $U\in\mathcal V_{(G,\tau)}(e_G)$ follows from Theorem \ref{semitop}, while the necessity of $[G,N_\sigma]\leq N_\tau$ follows from Lemma \ref{G,N<N}. The sufficiency of the two conditions is a consequence of Theorem \ref{semitop}.
\end{proof}

If $\tau$ is Hausdorff in this proposition, then (b) becomes $N \leq Z(G)$. So we have the following corollary, which can be also seen as a consequence of Proposition \ref{t2,z=1,t2}.

\begin{cor}
Let $G$ be a group. If $\tau$ is a Hausdorff group topology on $G$, then for every non-central $\tau$-open subgroup $N$ of $G$ $id_G: (G,\tau) \to (G,\zeta_N)$ is not semitopological.
\end{cor}

Combining together the two items of Proposition \ref{sigma-at} we have precisely the following corollary, which is the ``almost trivial version'' of Theorem \ref{semitop}. Furthermore it shows that the necessary condition of Lemma \ref{G,N<N} becomes also sufficient in the case of almost trivial topologies.

\begin{cor}\label{almdiscrsemitop}\emph{\cite[Lemma 5.15]{AGB}}
Let $G$ be a group and $\zeta_N\geq\zeta_L$ almost trivial group topologies on $G$. Then $id_G:(G,\zeta_N)\to(G,\zeta_L)$ is semitopological if and only if $[G,L]\leq N$.
\end{cor}

The next example is a consequence of this corollary.

\begin{exmp}\label{almdiscrex}
Let $G$ be a group and $\zeta_N$ an almost trivial group topology on $G$. Consider $$(G,\delta_G)\xrightarrow{id_G}(G,\zeta_N)\xrightarrow{id_G}(G,\iota_G).$$
Then:
\begin{itemize}
\item[(a)]$id_G:(G,\delta_G)\to(G,\zeta_N)$ is semitopological if and only if $N\leq Z(G)$;
\item[(b)]$id_G:(G,\zeta_N)\to(G,\iota_G)$ is semitopological if and only if $G'\leq N$.
\end{itemize}
\end{exmp}

On a group $G$ it is possible to consider the almost trivial topology generated by $G'$, that is $\zeta_{G'}$. A group $G$ is \emph{perfect} if $G=G'$, and $G$ is perfect if and only if $\zeta_{G'}=\iota_G$.

\begin{rem}\label{ind-remark}
With this topology generated by the derived group, we can write again Theorem \ref{ind}(b) as:
\begin{quote}
Let $G$ be a group and $\tau$ a group topology on $G$. Then $id_G:(G,\tau)\to(G,\iota_G)$ is semitopological if and only if $\tau\leq\zeta_{G'}$.
\end{quote}
\end{rem}

\begin{rem}
Let $N$ be a normal subgroup of a group $G$ and let $\zeta_N$ be the respective almost trivial topology on $G$.

Let $\tau$ be a group topology on $G$. Then $id_G:(G,\zeta_{N_\tau})\to (G,\tau)$ is continuous. Moreover, if $\zeta_L$ is another almost trivial topology on $G$ such that $id_G:(G,\zeta_{L})\to (G,\tau)$ is continuous, then $id_G:(G,\zeta_{L})\to (G,\zeta_{N_\tau})$ is continuous.

\begin{equation*}
\xymatrix{
(G,\zeta_{N_\tau}) \ar[rr] & & (G,\tau) \\
& (G,\zeta_L) \ar[ur] \ar@{-->}[ul] &
}
\end{equation*}

Consequently $id_G:(G,\zeta_{L})\to (G,\tau)$ semitopological implies $id_G:(G,\zeta_{L})\to (G,\zeta_{N_\tau})$ semitopological by Lemma \ref{primo_ok}, that is, $[G,N_\tau]\leq L$ by Corollary \ref{almdiscrsemitop}.
\end{rem}

\section{Ta\u\i manov groups}

Let $F\in[G]^{<\omega}$ be a finite subset of $G$ and $$c_G(F)=\bigcap_{x\in F}c_G(x)$$ the centralizer of $F$ in $G$. Then $\mathcal C=\{c_G(F):F\in[G]^{<\omega}\}$ is a family of subgroups of $G$ closed under finite intersections.
Then the Ta\u\i manov topology $T_G$ has $\mathcal C$ as local base at $e_G$, that is $\mathcal B_{T_G}=\mathcal C$.

\smallskip
We collect in the next lemma the first properties of this topology.

\begin{lem}\label{taimanov}
Let $G$ be a group. Then:
\begin{itemize}
	\item[(a)]$N_{T_G}=Z(G)$;
	\item[(b)]$T_G$ is Hausdorff if and only if $Z(G)=\{e_G\}$;
	\item[(c)]$G$ is abelian if and only if $T_G=\iota_G$;
	\item[(d)]in case $G$ is finitely generated, $T_G$ is almost trivial; in particular $T_G=\delta_G$ if and only if $Z(G)$ is trivial.
\end{itemize}
\end{lem}

\subsection{Permanence properties of the class $\mathfrak T$}

The following results show that the Ta\u\i manov topology has nice properties. The next proposition proves that it is a functorial topology with respect to continuous surjective homomorphisms.

\begin{prop}\label{functorial}
Let $G$ be a group. Then every surjective homomorphism $f:(G,T_G)\to (H,T_H)$ is continuous. 
\end{prop}
\begin{proof}
Let $h\in H$ and consider $g\in G$ such that $f(g)=h$. Then $f(c_G(g))\subseteq c_H(h)$. This proves the continuity of $f:(G,T_G)\to (H,T_H)$.
\end{proof}

On the other hand, the next example shows that the Ta\u\i manov topology is not functorial with respect to open surjective homomorphisms.

\begin{exmp}\label{H_Z}
For $$\mathbb H_\Z:=\begin{pmatrix}
1 & \Z & \Z \\
0 & 1 & \Z \\
0 & 0 & 1
\end{pmatrix}$$ the group of upper unitriangular $3\times 3$ matrices over $\Z$,
the canonical projection $\pi:(\mathbb H_\Z,T_{\mathbb H_\Z})\to (\mathbb H_\Z/Z(\mathbb H_\Z),T_{\mathbb H_\Z/Z(\mathbb H_\Z)})$ is not open.

Indeed, since $\mathbb H_\Z/Z(\mathbb H_\Z)=:G$ is abelian, $T_G=\iota_G$ by Lemma \ref{taimanov}(c). Moreover note that $G\cong \Z\times \Z$. Let $h=\begin{pmatrix}
1 & 0 & 0 \\
0 & 1 & 1 \\
0 & 0 & 1
\end{pmatrix}\in \mathbb H_\Z$. Then $c_{\mathbb H_\Z}(h)=\begin{pmatrix}
1 & 0 & \Z \\
0 & 1 & \Z \\
0 & 0 & 1
\end{pmatrix}$ and $\pi(c_{\mathbb H_\Z}(h))\cong \{0\}\times\Z$, which is not open in $(G,\iota_G)$.
\end{exmp}

\begin{lem}\label{tai-products}
Let $G=\prod_{i\in I}G_i$. Then $\prod_{i\in I}T_{G_i}\leq T_G$. If $I=\{1,\ldots,n\}$ is finite, then $T_{G_1}\times\ldots\times T_{G_n}= T_G$.
\end{lem}
\begin{proof}
Since all the canonical projections $\pi_i:(G,T_G)\to(G_i,T_{G_i})$ are continuous by Proposition \ref{functorial}, $\prod_{i\in I}T_{G_i}\leq T_G$.

Suppose now that $I=\{1,\ldots,n\}$ is finite. If $F$ is a finite subset of $G_1\times\ldots \times G_n$, then it is contained in a finite subset of the form $F_1\times\ldots \times  F_n$, where each $F_i$ is a finite subset of $G_i$ for $i=1,\ldots, n$. Moreover $c_G(F_1\times\ldots \times F_n)=c_{G_1}(F_1)\times\ldots \times c_{G_n}(F_n)$. This proves that $T_G=T_{G_1}\times \ldots \times T_{G_n}$.
\end{proof}

\begin{prop}\label{zg=1}
\begin{itemize}
	\item[(a)]If $G\in\mathfrak T$, then $Z(G)=\{e_G\}$.
	\item[(b)]If $G\in\mathfrak T_t$, then $G$ is perfect.
\end{itemize}
\end{prop}
\begin{proof}
(a) Follows from Lemma \ref{taimanov}(b).

(b) Since $G/G'$ is abelian and in $\mathfrak T$, $G/G'$ is trivial in view of Lemma \ref{taimanov}(c), that is $G=G'$.
\end{proof}

It follows from (a) that every non-trivial abelian group $G\not\in\mathfrak T$.

\medskip
The next result about products is a consequence of Lemma \ref{tai-products}.

\begin{prop}\label{t-products}
The class $\mathfrak T$ is closed under taking finite products.
\end{prop}
\begin{proof}
Let $G_1,G_2\in\mathfrak T$ and $G:=G_1\times G_2$. By Lemma \ref{tai-products} $T_G=T_{G_1}\times T_{G_2}$ and so $T_G=\delta_G$, that is $G\in\mathfrak T$. This can be extended to all finite products.
\end{proof}

The next example in particular shows that $\mathfrak T$ is not closed under taking quotients and subgroups since the groups in (b) and (c) have abelian quotients (so they are not in $\mathfrak T_t$) and non-trivial abelian subgroups.

\begin{exmp}\label{finite-taimanov}
\begin{itemize}
	\item[(a)]A finite group $G\in\mathfrak T$ if and only if $Z(G)=\{e_G\}$. This follows from Lemma \ref{taimanov}, but can be also simply directly proved.
	\item[(b)] Let $G=\begin{pmatrix}
\mathbb R^* & \mathbb R \\
0 & 1\end{pmatrix}$. Then $G\in\mathfrak T$. Indeed, for $F=\left\{\begin{pmatrix}
2 & 0 \\
0 & 1\end{pmatrix},\begin{pmatrix}
1 & 1 \\
0 & 1\end{pmatrix}\right\}$ $c_G(F)=\{e_G\}$.
	\item[(c)] Every non-abelian free group $F(X)$ of rank $>1$ is in $\mathfrak T$. Indeed, for $F=\{a,b\}$, where $a,b\in X$ are generators of $F(X)$, $c_{F(X)}(F)=\{e_{F(X)}\}$.
\end{itemize}
\end{exmp}

This example shows also that the condition ``surjective'' in Proposition \ref{functorial} cannot be removed: if $G$ is one of the groups in (b) or (c), then $G$ has some non-trivial abelian subgroup $H$. Since $H$ is abelian, $T_H=\iota_H$, while $T_G=\delta_G$. Consequently the injective homomorphism $(H,T_H)\to (G,T_G)$ is far from being continuous.

\begin{rem}
Since non-abelian free groups of rank $>1$ are Ta\u\i manov (as shown in Example \ref{finite-taimanov}(c)),
\begin{itemize}
	\item there exist arbitrarily large Ta\u\i manov groups; moreover, every non-abelian subgroup of a non-abelian free group is Ta\u\i manov, being free \cite{Rob};
	\item every group is quotient of a Ta\u\i manov group, since every group is quotient of a non-abelian free group of rank $>1$ \cite{Rob}.
\end{itemize}
\end{rem}

It is not clear if this holds also for subgroups:

\begin{question}
Is every group subgroup of a Ta\u\i manov group?
\end{question}

Theorem \ref{abelian(->tai} answers positively the question in the abelian case.

For an abelian group $G$ and $p\in\mathbb P$ in what follows we denote by $r_p(G)$ the $p$-rank of $G$.

\begin{lem}\label{r2=0}
Let $G$ be an abelian group with $r_2(G)=0$. Then there exists $H\in\mathfrak T$ such that $G\leq H$ and $[H:G]=2$.
\end{lem}
\begin{proof}
Let $f:G\to G$ be defined by $f(x)=-x$ for every $x\in G$. Moreover let $$H:=G\rtimes \langle f\rangle$$ ($\rtimes$ denotes the semidirect product). Then $c_H(0,f)=\langle (0,f)\rangle$ and $(0,f)\not\in c_H(g,id_G)$ for every $g\in G\setminus\{0\}$. Consequently for $F=\{(g,id_G),(0,f)\}$, with $g\in G\setminus\{0\}$, $c_H(F)=\{e_H\}$, that is $H\in\mathfrak T$. Since $f$ has order $2$, $G$ has index $2$ in $G$.
\end{proof}

\begin{theo}\label{abelian(->tai}
For every abelian group $G$ there exists a group $H\in\mathfrak T$ containing $G$ as a subgroup and such that $|H|=\omega\cdot |G|$.
\end{theo}
\begin{proof}
Let $G$ be an abelian group. Then $G\subseteq D(G)=G_1\oplus G_2$, where $r_2(G_1)=0$ and $r_2(G_2)=r_2(D(G))$. Then there exists $H_1\in\mathfrak T$ such that $G_1\leq H_1$ and $|H_1|=2\cdot |G_1|$ by Lemma \ref{r2=0}.

Now consider $G_2=\bigoplus_{r_2(G)}\Z(2^\infty)$. If $r_2(G)\leq\omega$, then $G_2$ is contained in $\bigoplus_\omega\Z(2^\infty)$. Let then $\sigma$ be the shift $\bigoplus_\omega\Z(2^\infty)\to \bigoplus_\omega\Z(2^\infty)$ defined by $(x_n)_n\mapsto(x_{n-1})_n$ for every $(x_n)_n\in \bigoplus_\omega\Z(2^\infty)$. Then $\sigma^n$ has no non-zero fixed point for every $n\in\Z$, $n\neq 0$.

\bigskip
\noindent\textbf{Claim.}
Let $G$ be an abelian group and let $f$ be an automorphism of $G$ such that $f^n$ has no non-zero fixed point for every $n\in\Z$, $n\neq 0$. Then there exists $H\in\mathfrak T$ such that $G\leq H$ and $|H|=\omega\cdot|G|$.
\begin{proof}[Proof of the claim.]
Let $$H:=G\rtimes \langle f\rangle.$$ Then $c_H(0,f)=\langle (0,f)\rangle$ and $(0,f)\not\in c_H(g,id_G)$ for every $g\in G\setminus\{0\}$. Consequently for $F=\{(g,id_G),(0,f)\}$, with $g\in G\setminus\{0\}$, $c_H(F)=\{e_H\}$, that is $H\in\mathfrak T$. Since $f$ has infinite order $|H|=\omega\cdot|G|$.
\end{proof}

By the claim there exists a group $H_2\in\mathfrak T$ such that $G_2\leq \bigoplus_\omega\Z(2^\infty)\leq H_2$ and $|H_2|=\omega$. Suppose that $r_2(G)\geq\omega$. Then $G_2\cong \bigoplus_{r_2(G)} (\bigoplus_\omega\Z(2^\infty))$. Let $\widetilde \sigma:\bigoplus_{r_2(G)} (\bigoplus_\omega\Z(2^\infty))\to \bigoplus_{r_2(G)} (\bigoplus_\omega\Z(2^\infty))$ be defined by $\widetilde \sigma\restriction_{\bigoplus_\omega\Z(2^\infty)}=\sigma$. Then $\widetilde\sigma^n$ has no non-zero fixed point for every $n\in\Z$, $n\neq 0$, and again the claim gives a group $H_2\in\mathfrak T$ that contains $G_2$ as a subgroup and such that $|H_2|=|G_2|$.

Let $H:=H_1\oplus H_2$. By Proposition \ref{t-products} $H\in\mathfrak T$. Moreover $|H|=\omega\cdot|H_1|\cdot |H_2|=\omega\cdot|G_1|\cdot |G_2|=\omega\cdot|G|$.
\end{proof}

Lemma \ref{l2} shows that to prove that a group is Ta\u\i manov it suffices to consider a convenient quotient with a finite normal subgroup and check whether it is Ta\u\i manov.

\begin{claim}\label{l1}
Let $G$ be a group with $Z(G)=\{e_G\}$. If there exists a finite subset $F$ of $G$ such that $c_G(F)$ is finite, then there exists another finite subset $F_1\supseteq F$ of $G$ such that $c_G(F_1)=\{e_G\}$. In particular $G\in\mathfrak T$.
\end{claim}
\begin{proof}
Let $c_G(F)=\{e_G,g_1,\ldots,g_n\}$. Since $Z(G)=\{e_G\}$, for every $i\in\{1,\ldots,n\}$ there exists $h_i\in G$ such that $[g_i,h_i]\neq e_G$; in particular $g_i\not \in c_G(h_i)$. Let $F_1=F\cup\{h_1,\ldots,h_n\}$. Then $g_i\not\in c_G(F_1)$ for every $i\in\{1,\ldots,n\}$. Since $c_G(F_1)\subseteq c_G(F)=\{e_G,g_1,\ldots, g_n\}$, this proves that $c_{G}(F_1)=\{e_G\}$.
\end{proof}

\begin{lem}\label{l2}
Let $G$ be a group with $Z(G)=\{e_G\}$ and let $N$ be a normal finite subgroup of $G$ such that $G/N\in \mathfrak T$. Then $G\in\mathfrak T$.
\end{lem}
\begin{proof}
Let $F_1$ be a finite subset of $G/N$ such that $c_{G/N}(F_1)=\{e_{G/N}\}$. Let $\pi:G\to G/N$ be the canonical projection and let $F$ be a finite subset of $G$ such that $\pi(F)=F_1$. Since $\pi(c_G(F))\subseteq c_{G/N}(F_1)=\{e_{G/N}\}$, $c_G(F)\subseteq N$. Since $N$ is finite, Claim \ref{l1} applies to conclude that $G\in\mathfrak T$.
\end{proof}

The next is an example of a totally Ta\u\i manov group.

\begin{exmp}\label{SO_3(R)inT}
We denote by $G:=SO_3(\mathbb R)$ the group of all orthogonal matrices $3\times 3$ with determinant $1$ and coefficients in $\mathbb R$. Then $G\in\mathfrak T$. Since $G$ is simple, $G\in\mathfrak T_t$.

\smallskip
Indeed, $G=\bigcup_\alpha T_\alpha$, where $T_\alpha\cong\mathbb T$ and $T_\alpha$ is generated by an element $\alpha$ of $G$, that is, $\overline{\langle\alpha\rangle}=T_\alpha$. Moreover $c_{G}(\alpha)$ contains $T_\alpha$ as a finite index subgroup and for $\alpha,\beta\in G$ with $\alpha\neq\beta$ and $T_\alpha\neq T_\beta$, $T_\alpha\cap T_\beta$ is finite. Then $c_{G}(\alpha)\cap c_{G}(\beta)$ is finite. By Claim \ref{l1} $G\in\mathfrak T$.
\end{exmp}

\subsection{The permutations groups}

For a set $X$, $x\in X$ and a subgroup $H$ of $S(X)$ let 
\begin{align*}
&O_H(x):=\{h(x):h\in H\},\\ 
&\text{Stab}\, x:=\{\rho\in S(X): \rho(x)=x\},\ \text{and}\\ 
&S_x:=\text{Stab}\, x\cap H.
\end{align*}
Moreover $H$ induces a partition of $X$, that is $X=\bigcup_{x\in R_H}O_H(x)$, where $R_H\subseteq X$ is a set of representing elements.

If $\tau\in S(X)$, then $$\text{Stab}\,x=(\text{Stab}\,\tau(x))^\tau.$$

\begin{rem}\label{claim-sx}
Let $X$ be a set and $H$ a subgroup of $S(X)$. If $\tau\in N_{S(X)}(H)$, then:
\begin{itemize}
	\item[(a)] $\tau(O_H(x))=O_H(\tau(x))$;
	\item[(b)] $S_x=(S_{\tau(x)})^\tau$; indeed, $S_x=\text{Stab}\, x\cap H=(\text{Stab}\,\tau(x))^\tau\cap H^\tau=(\text{Stab}\,\tau(x)\cap H)^\tau=(S_{\tau(x)})^\tau$;
	\item[(c)] $\tau$ induces a permutation $\widetilde\tau$ of $R_H$. Indeed, $\tau(O_H(x))=O_H(\tau(x))$ by (a); so we can define $\widetilde\tau(x)=y$, where $y\in R_H$ is the representing element of $O_H(\tau(x))$. Then $$c_{S(X)}(H)=\left\{\tau\in\bigcap_{x\in R_H\setminus \text{supp}\, \widetilde\tau}N_H(S_x)\cdot \text{Stab}\, x: [H,\tau]\subseteq \bigcap_{x\in\text{supp}\, \widetilde\tau}S_x\right\}.$$
\end{itemize}
\end{rem}

We describe the subgroups of $S(X)$ with trivial centralizer:

\begin{lem}\label{c_S(H)}
Let $X$ be a set and $H$ a subgroup of $S(X)$. Then $c_{S(X)}(H)=\{id_X\}$ if and only if the following conditions hold:
\begin{itemize}
	\item[(a)]$S_x=N_H(S_x)$ for every $x\in R_H$, and
	\item[(b)]$S_x$ and $S_y$ are not conjugated in $H$ for every $x,y\in R_H$ with $x\neq y$. 
\end{itemize}
\end{lem}
\begin{proof}
Let $\tau\in c_{S(X)}(H)\setminus\{id_X\}$. There exists $x\in R_H$ such that $y:=\tau(x)\neq x$. Indeed, if $\tau(x)=x$ for every $x\in R_H$, then for every $z\in X$, there exist $h\in H$ and $x\in R_H$ such that $z=h(x)$, and so $\tau(z)=\tau(h(x))=h(\tau(x))=h(x)=z$. By Remark \ref{claim-sx}(a,b) $$\tau(O_H(x))=O_H(\tau(x))\ \text{and}\ S_x=(S_y)^\tau=S_y.$$

If $y\in O_H(x)$, then $\tau\restriction_{O_H(x)}:O_H(x)\to O_H(x)$ is a bijection and $y=h_0(x)$ for some $h_0\in H$; then $\tau(h(x))=h(\tau(x))=h h_0(x)$ for every $h\in H$. Let $h\in S_x$. Since $\tau$ is well-defined, $h(x)=x$ implies $h h_0 (x)=h_0(x)$, that is $(h_0)^{-1} h h_0(x)=x$. This is equivalent to $h^{h_0}\in S_x$, that is $h_0\in N_H(S_x)$. But $h_0\not\in S_x$ and this contradicts (a).

Suppose now that $y\not\in O_H(x)$ and so $O_H(x)\cap O_H(y)=\emptyset$. Let $z\in R_H\cap O_H(y)$. Then $y=h_0(z)$ for some $h_0\in H$. By Remark \ref{claim-sx}(b) $S_z=(S_y)^{h_0}=(S_x)^{h_0}$ and this contradicts (b). 

\smallskip
Assume that there exists $h_0\in N_H(S_x)\setminus S_x$ for some $x\in R_H$. Let $\tau:X\to X$ be defined by $\tau(x)=h_0(x)$, $\tau(h(x))=h h_0(x)$ for every $h\in H$ and $\tau(y)=y$ for every $y\in X\setminus O_H(x)$. This $\tau$ is well-defined. Indeed, if $h_1(x)=h_2(x)$ for some $h_1,h_2\in H$, that is, $h_2^{-1}h_1\in S_x$; then $h_1 h_0(x)=h_2 h_0(x)$, equivalently $h_0^{-1} (h_2^{-1}h_1) h_0(x)=x$, that is, $h_0^{-1} (h_2^{-1}h_1) h_0\in S_x$, which holds true by the hypothesis that $h_0\in N_H(S_x)$. Moreover, it is possible to check that $\tau\in S(X)$. By the definition $\tau h= h \tau$ for every $h\in H$ and so $id_X\neq\tau \in c_{S(X)}(H)$.

Suppose that $S_x=(S_z)^{h_0}$ for some $x,z\in R_H$ and $h_0\in H$. Then for $y=h_0^{-1}(z)\in O_H(z)$ we have $S_y=(S_z)^{h_0}=S_x$ by Remark \ref{claim-sx}(b). Define $\tau:X\to X$ as $\tau(x)=y$, $\tau(h(x))=h(y)$ for every $h\in H$ and $\tau(w)=w$ for every $w\in X\setminus O_H(x)$. Then $\tau$ is well-defined; indeed, if $h_1(x)=h_2(x)$ for some $h_1,h_2\in H$, that is, $h_2^{-1} h_1\in S_x$, then $h_1(y)=h_2(y)$, equivalently, $h_2^{-1} h_1\in S_y$, which holds true since $S_x=S_y$. Moreover, it is possible to check that $\tau\in S(X)$. By the definition $\tau h= h \tau$ for every $h\in H$ and so $id_X\neq\tau \in c_{S(X)}(H)$.
\end{proof}

\begin{prop}\label{family}
For a cardinal $\kappa$ the following conditions are equivalent:
\begin{itemize}
	\item[(a)]there exists a set $X$ with $|X|=\kappa$ and $S(X)\in\mathfrak T$;
	\item[(b)]there exists a set $X$ with $|X|=\kappa$ such that there exists a finitely generated subgroup $H$ of $S(X)$ such that $S_x=N_H(S_x)$ for every $x\in R_H$ and $S_x$, $S_y$ are not conjugated for every $x,y\in R_H$ with $x\neq y$.
\end{itemize}
If $\kappa>\omega$, then the following condition is equivalent to the previous:
\begin{itemize}
	\item[(c)]there exists a finitely generated group $H$ admitting a family $\mathfrak S=\{S_\alpha:\alpha<\kappa\}$ of subgroups of $H$ such that $S_\alpha=N_H(S_\alpha)$ for every $\alpha<\kappa$.
\end{itemize}
\end{prop}
\begin{proof}
(a)$\Leftrightarrow$(b) The condition $S(X)\in\mathfrak T$ is equivalent to the existence of a finite subset $F$ of $S(X)$ such that $c_{S(X)}(F)=\{id_X\}$. Let $H=\langle F\rangle$. Then $$c_{S(X)}(H)=c_{S(X)}(F)$$ and so equivalently $c_{S(X)}(H)=\{id_X\}$. By Lemma \ref{c_S(H)} we have the conclusion.

(b)$\Rightarrow$(c) Since $\kappa>\omega$, and each $O_H(x)$ is countable, $|R_H|=\kappa$. So $\{S_x:x\in R_H\}$ is the family requested in (c).

(c)$\Rightarrow$(b) Since $\kappa>\omega$ and $H$ is countable, we can suppose that $\mathfrak S$ has the property that $S_\alpha$ and $S_\beta$ are not conjugated in $H$ for every $\alpha,\beta<\kappa$ with $\alpha\neq\beta$. Indeed, every subgroup $S_\alpha$ of $H$ has at most countably many conjugated subgroups in $H$, so we can restrict the family $\mathfrak S$ taking only one element for every class of conjugation, finding a subfamily of the same cardinality $\kappa$ as $\mathfrak S$.

Define $X_\alpha:=\{h S_\alpha:h\in H\}$ for every $\alpha<\kappa$ and $X:=\bigcup_{\alpha<\kappa}X_\alpha$. Moreover let $x_\alpha:=id_H S_\alpha\in X_\alpha$ for every $\alpha>\kappa$. In particular $|X|=\kappa$. Moreover $H$ acts on $X$ by multiplication on the left and $O_H(x_\alpha)=X_\alpha$ for every $\alpha<\kappa$. There exists a group homomorphism $\varphi:H\to S(X)$; let $\widetilde H:=\varphi(H)\leq S(X)$.
Then $\widetilde H$ is finitely generated and the action of $\widetilde H$ on $X$ is the same as the action of $H$ on $X$. Then $O_{\widetilde H}(x_\alpha)=X_\alpha$ for every $\alpha<\kappa$ and $R_{\widetilde H}=\{x_\alpha:\alpha<\kappa\}$. Moreover $\varphi(S_\alpha)=\text{Stab}\, x_\alpha\cap \widetilde H=:S_{x_\alpha}$. Since $S_\alpha=N_H(S_\alpha)$ for every $\alpha<\kappa$ and $S_\alpha$ and $S_\beta$ are not conjugated for every $\alpha<\beta<\kappa$, it is possible to prove that $S_{x_\alpha}=N_H(S_{x_\alpha})$ for every $x_\alpha\in R_{\widetilde H}$ and $S_{x_\alpha}$ and $S_{x_\beta}$ are not conjugated for every $x_\alpha,x_\beta\in R_{\widetilde H}$ with $x_\alpha\neq x_\beta$. So the properties in (b) are satisfied.
\end{proof}

\begin{theo}\label{S(X)_Taimanov}
Let $X$ be a set with $|X|>2$.
\begin{itemize}
	\item[(a)]If $|X|\leq\omega$, then $S(X)\in\mathfrak T$.
	\item[(b)]If $|X|>\mathfrak c$, then $S(X)	\not\in\mathfrak T$.
\end{itemize}
\end{theo}
\begin{proof}
(a) Assume that $2<|X|<\omega$. Since $Z(S(X))$ is trivial, $S(X)\in\mathfrak T$ by Example \ref{finite-taimanov}(a).

Assume that $|X|=\omega$. We can suppose $X=\mathbb Z$.
Let $H=\langle \sigma,\tau\rangle$, where $\tau=(-1,1)$ and $\sigma$ is the shift, that is $\sigma(n)=n+1$ for every $n\in\mathbb Z$. Then $O_H(0)=\Z$ and so $R_H=\{0\}$. Moreover $S_0=\langle\tau\rangle$ and hence $N_H(S_0)=S_0$. By Proposition \ref{family} $S(X)\in\mathfrak T$.

(b) Let $H$ be a finitely generated subgroup of $S(X)$. Since $O_H(x)$ is countable for every $x\in R_H$, $|R_H|=|X|>\mathfrak c$.
Since $H$ is countable, it has at most $\mathfrak c$ subgroups and so there exists a subset $S$ of $R_H$ such that $|S|>\mathfrak c$ and $S_x=S_y$ for every $x,y\in S$. By Proposition \ref{family} $S(X)\not\in\mathfrak T$.
\end{proof}

\begin{question}\label{S(X)inT?}
Let $X$ be a set.
\begin{itemize}
	\item[(a)] Is $S(X)\in\mathfrak T$ if $|X|=\omega_1$?
	\item[(b)] Is $S(X)\in\mathfrak T$ if $|X|=\mathfrak c$?
	\item[(c)] Is $S(X)\in\mathfrak T$ if $\omega<|X|\leq\mathfrak c$?
\end{itemize}
\end{question}

\begin{rem}
Question \ref{S(X)inT?} can be formulated in equivalent terms thanks to Proposition \ref{family}. Indeed, if $X$ is a set of cardinality $\kappa$ with $\omega<\kappa\leq\mathfrak c$, then $S(X)\in\mathfrak T$ if and only if there exists a finitely generated group $H$ admitting a family $\mathfrak S=\{S_\alpha:\alpha<\kappa\}$ of subgroups of $H$ such that $S_\alpha=N_H(S_\alpha)$ for every $\alpha<\kappa$.

So Question \ref{S(X)inT?} becomes: does there exists a finitely generated group $H$ with a ``large'' (i.e., of cardinality $\kappa$ with $\omega<\kappa\leq \mathfrak c$) family of self-normalizing subgroups?
\end{rem}

\section{Problem A}\label{pbA}

We start considering a stability property of the the class $\mathfrak A$ of Arnautov groups.

\begin{theo}\label{A-quotients}
The class $\mathfrak A$ is closed under taking quotients.
\end{theo}
\begin{proof}
Let $G\in\mathfrak A$ and let $N$ be a normal subgroup of $G$. Let $\sigma\leq\tau$ be group topologies on $G/N$ such that $id_{G/N}:(G/N,\tau)\to (G/N,\sigma)$ is semitopological. Then $id_G:(G,\tau_i)\to (G,\sigma_i)$ is semitopological by Lemma \ref{initial}. Since $G\in\mathfrak A$, $\tau_i=\sigma_i$ and hence $\tau=\sigma$.
\end{proof}

In Section \ref{totally_Markov} we will comment the stability of $\mathfrak A$ under taking subgroups and products.

\begin{exmp}\label{Ex_x}
\begin{itemize}
	\item[(a)]Obviously every indiscrete group $G$ is A-complete.
	\item[(b)]Let $G$ be a group. Let $G_{ab}=G/G'$ be the abelianization of $G$ and endow $G_{ab}$ with the discrete topology and with the indiscrete topology:
\begin{equation*}
\begin{CD}
(G,\zeta_{G'}) @>>> (G_{ab},\delta_{G_{ab}})\\ @V{id_G}VV
@VV{id_{G_{ab}}}V\\ (G,\iota_G) @>>> (G_{ab},\iota_{G_{ab}})
\end{CD}
\end{equation*}
If $G\neq G'$ then $id_{G_{ab}}$ is a semitopological non-open isomorphism, because $G_{ab}$ is abelian, and $id_G$ is a semitopological non-open isomorphism too, in view of Remark \ref{ind-remark}. So $(G,\zeta_{G'})$ is not A-complete.
	\item[(c)]An abelian topological group $G$ is A-complete if and only if $G$ is indiscrete. In particular the only abelian Arnautov group is $G=\{e_G\}$ (as $(G,\delta_G)$ must be indiscrete).
\end{itemize}
\end{exmp}

The next proposition generalizes the example in (b).

\begin{prop}
A topological group $G$ with indiscrete derived group $G'$ is A-complete precisely when $G$ is indiscrete. 
\end{prop}
\begin{proof}
The conclusion follows from Remark \ref{ind-remark}.
\end{proof}

\begin{exmp}\label{sin-example}
Let $G$ be a group and $\tau$ a group topology on $G$.
\begin{itemize}
	\item[(a)]If $(G,\tau)$ is SIN, then it is A-complete if and only if for every group topology $\sigma<\tau$ on $G$ there exist $U\in\mathcal V_{(G,\tau)}(e_G)$ and $g\in G$ such that $[g,V_g]\not\subseteq U$ for every $V_g\in\mathcal V_{(G,\sigma)}(e_G)$ (this follows from Proposition \ref{SINsemitop}).
	\item[(b)]If $(G,\tau)$ is Hausdorff and $\tau\leq\zeta_{G'}$ (as already noted after Theorem \ref{ind}, this condition yields $\tau$ SIN), then $G$ is abelian and consequently $\tau>\iota_G$ implies that $(G,\tau)$ is not A-complete (supposing that $G$ is not a singleton).
\end{itemize}
\end{exmp}

\begin{prop}\label{initial-minimal}
Let $G$ be a group and $N$ a normal subgroup of $G$. Let $\tau$ be a group topology on $G/N$ and $\tau_i$ the initial topology of $\tau$ on $G$. Then $\tau$ is A-complete if and only if $\tau_i$ is A-complete.
\end{prop}
\begin{proof}
Let $id_{G/N}:(G/N,\tau)\to (G/N,\sigma)$ be semitopological, where $\sigma\leq\tau$ is another group topology on $G$. By Lemma \ref{initial} also $id_G:(G,\tau_i)\to (G,\sigma_i)$ is semitopological and the hypothesis implies that $\tau_i=\sigma_i$. Consequently $\tau=\sigma$.

Suppose that $\tau$ is A-complete.
Let $\sigma<\tau_i$ be another group topology on $G$ and consider the quotient topology $\sigma_q$ of $\sigma$ on $G/N$. So we have the following situation:
\begin{equation*}
\begin{CD}
(G,\tau_i) @>{id_G}>> (G,\sigma)\\
@V{\pi}VV @VV{\pi}V\\
(G/N,\tau) @>{id_{G/N}}>> (G/N,\sigma_q).
\end{CD}
\end{equation*}

Since $\sigma<\tau_i$, it follows that $N_\sigma\geq N_{\tau_i}=N$. Consequently $\sigma$ is the initial topology of $\sigma_q$ and so $\sigma_q<\tau$, otherwise $\sigma=\tau_i$.
By hypothesis $id_{G/N}:(G/N,\tau)\to(G/N,\sigma_q)$ is not semitopological. To conclude that also $id_G:(G,\tau_i)\to(G,\sigma)$ is not semitopological apply Theorem \ref{quotientst}.
\end{proof}

\begin{cor}\label{t_q->t_A-complete}
Let $G$ be a group and $\tau$ a group topology on $G$. Consider the quotient $G/N_\tau$ and the quotient topology $\tau_q$ of $\tau$ on $G/N_\tau$. Then $\tau$ is A-complete if and only if $\tau_q$ is A-complete.
\end{cor}
\begin{proof}
Since $\tau$ is the initial topology of $\tau_q$, it suffices to apply Proposition \ref{initial-minimal}.
\end{proof}

Now we give a necessary condition for a group to be Arnautov.

\begin{prop}\label{perfect}
For a group $G$ the following conditions are equivalent:
\begin{itemize}
\item[(a)] $id_G:(G,\tau)\to(G,\iota_G)$ is semitopological for no group topology $\tau>\iota_G$ on $G$;
\item[(b)]$G$ is perfect.
\end{itemize}
\end{prop}
\begin{proof}
(a)$\Rightarrow$(b) Since $id_G:(G,\zeta_{G'})\to(G,\iota_G)$ is a semitopological isomorphism by Theorem \ref{ind}(b), our hypothesis (a) implies $\zeta_{G'}=\iota_G$ and hence $G=G'$. 

(b)$\Rightarrow$(a) Suppose $G=G'$; then $\zeta_{G'}=\iota_G$. If $id_G:(G,\tau)\to(G,\iota_G)$ is a semitopological isomorphism, then $\tau\leq\zeta_{G'}=\iota_G$ by Theorem \ref{ind}(b), so $\tau=\iota_G$. This means that $id_G$ is open.
\end{proof}

Therefore, if a group $G$ is Arnautov, then for every non-indiscrete group topology $\tau$ on $G$ $id_G:(G,\tau)\to(G,\iota_G)$ is not semitopological. In particular Proposition \ref{perfect} implies that every Arnautov group is perfect.

\begin{cor}\label{Coro}
Let $G$ be a simple non-abelian group and $\tau$ a group topology on $G$. If $\tau>\iota_G$, then $id_G:(G,\tau)\to (G,\iota_G)$ is not semitopological.
\end{cor}

A consequence of these results is that every minimal simple non-abelian group $(G,\tau)$ is A-complete. Indeed, if $\sigma\leq\tau$ is another group topology on $G$ and $id_G:(G,\tau)\to(G,\sigma)$ is semitopological, then by Lemma \ref{simplenonabrem} either $\sigma$ is Hausdorff or $\sigma=\iota_G$. Since $G$ is simple and non-abelian, $G$ is perfect. Then Proposition \ref{perfect} implies that $\sigma$ is not indiscrete and so $\sigma$ has to be Hausdorff. The minimality of $\tau$ yields that $\sigma=\tau$.

This consequence is improved by the next result.

\begin{prop}\label{min,z01,a-min}
If $(G,\tau)$ is a minimal group and $Z(G)=\{e_G\}$, then $(G,\tau)$ is A-complete.
\end{prop}
\begin{proof}
Let $\sigma\leq\tau$ be a group topology on $G$ and suppose that $id_G:(G,\tau)\to(G,\sigma)$ is semitopological. By Proposition \ref{t2,z=1,t2} $\sigma$ is Hausdorff and so $\sigma=\tau$ by the minimality of $\tau$.
\end{proof}

\begin{exmp}\label{simple-finite}
	Every simple finite non-abelian group $G$ is an Arnautov group. Indeed, the only group topologies on $G$ are the trivial ones and $id_G:(G,\delta_G)\to(G,\iota_G)$ is not semitopological by Corollary \ref{Coro}.
\end{exmp}

The following remark could be used as a test to verify if a group is Arnautov.

\begin{rem}
If $G\in\mathfrak A$, then for every group topology $\tau$ on $G$ and for every normal subgroup $N$ of $G$,
\begin{itemize}
	\item $id_G: (G,\sup\{\tau,\zeta_N\})\to (G,\tau)$ is not semitopological if $\sup\{\tau,\zeta_N\}>\tau$;
	\item $id_G:(G,\sup\{\tau,\zeta_N\})\to(G,\zeta_N)$ is not semitopological if $\sup\{\tau,\zeta_N\}>\zeta_N$.
\end{itemize}
\end{rem}

\subsection{When the discrete topology is A-complete}

\begin{rem}\label{discrhom}{\cite[Corollary 5.3]{AGB}}
We can formulate Theorem \ref{discrete}(a) in terms of the Ta\u\i manov topology:
\begin{quote}
{Let $G$ be a group and $\sigma$ a group topology on $G$. Then $id_G:(G,\delta_G)\to (G,\sigma)$ is semitopological if and only if $\sigma\geq T_G$, that is, $N_\sigma\leq N_{T_G}=Z(G)$.}
\end{quote}
\end{rem}

Consequently the Ta\u\i manov topology is the coarsest topology $\sigma$ on a group $G$ such that $id_G: (G,\delta_G) \to (G,\sigma)$ is semitopological. So, since in this section we consider the case when the discrete topology is A-complete, we have to impose that the Ta\u\i manov topology is discrete, that is, the group is Ta\u\i manov. This also motivates Definition \ref{tai-def}.

\smallskip
The next theorem solves a particular case of Problem A, that is, it characterizes the groups for which the discrete topology is A-complete.

\begin{theo}\label{discrAminz1}
Let $G$ be a group. Then $\delta_G$ is A-complete if and only if $G\in\mathfrak T$.
\end{theo}
\begin{proof} 
Suppose that $\delta_G>T_G$. Then $id_G:(G,\delta_G)\to (G,T_G)$ is semitopological by Remark \ref{discrhom}. This proves that $\delta_G$ is not A-complete. Suppose that $\delta_G=T_G$. Let $\tau<\delta_G$ be a group topology on $G$. Then $id_G:(G,\delta_G)\to (G,\tau)$ is not semitopological by Remark \ref{discrhom}. This proves that $\delta_G$ is A-complete.
\end{proof}

By Proposition \ref{zg=1}(a) the equivalent conditions of this theorem imply that the group has trivial center. The next example shows that they can be strictly stronger than having trivial center. Moreover this is an example of a Ta\u\i manov group which has an infinite non-abelian subgroup that is not Ta\u\i manov.

\begin{exmp}
Consider $S(\mathbb N)$ and let $G:=S_\omega(\mathbb N)$ be the subgroup of $S(\mathbb N)$ of the permutations with finite support, that is $S_\omega=\bigcup_{n=1}^\infty S_n$. Then $Z(G)=\{e_G\}$. If $F$ is a finite subset of $G$, then there exists $n\in\mathbb N_+$ such that $F\subseteq S_n$ and $c(S_n)=S(\mathbb N\setminus\{1,\ldots,n\})$ is infinite. Therefore $T_G<\delta_G$ and so $G\not\in \mathfrak T$.
\end{exmp}

Anyway in the finite case the three conditions are equivalent, as stated by Example \ref{finite-taimanov}(a).

\medskip
The next theorem characterizes the almost trivial topologies that are A-complete. It covers Theorem \ref{discrAminz1}.

\begin{theo}\label{almdiscr-solution}
Let $G$ be a group and $N\triangleleft G$. Then $(G,\zeta_N)$ is A-complete if and only if $G/N\in\mathfrak T$.
\end{theo}
\begin{proof}
Suppose that $\zeta_N$ is A-complete. Since $\zeta_N$ is the initial topology of $\delta_{G/N}$, it follows that $\delta_{G/N}$ is A-complete by Proposition \ref{initial-minimal}. By Theorem \ref{discrAminz1} this is equivalent to $G/N\in\mathfrak T$.

Suppose now that $G/N\in\mathfrak T$. By Theorem \ref{discrAminz1} this is equivalent to say that $\delta_{G/N}$ is A-complete and so $\zeta_N$ is A-complete by Corollary \ref{t_q->t_A-complete}.
\end{proof}

The next theorem offers a relevant necessary condition for a group to be Arnautov:

\begin{theo}\label{perfect-taimanov}
If $G\in\mathfrak A$, then $G\in\mathfrak T_t$.
\end{theo}
\begin{proof}
The conclusion follows from Theorems \ref{A-quotients} and \ref{discrAminz1}.
\end{proof}

So the next question naturally arises.

\begin{question}\label{perfect-taimanov_question}
Does $G\in\mathfrak T_t$ imply $G\in\mathfrak A$?
\end{question}

We shall give a positive answer to this question in a particular case in Proposition \ref{totM,A-complete,Arnautov}.

\medskip
The next examples show that a group can admit two A-complete topologies that are one strictly finer than the other.

\begin{exmp}\label{S(Z)}
Let $G:=S(\mathbb Z)$ and $S:=S_\omega(\Z)>A:=A_\omega(\Z)$, which are the only proper normal subgroups of $G$.
\begin{itemize}
	\item[(a)]The point-wise convergence topology $\mathcal T$ on $G$ is A-complete: $\mathcal T$ is minimal and $Z(G)$ is trivial, so Proposition \ref{min,z01,a-min} applies.
	\item[(b)]The discrete topology $\delta_G$ is A-complete by Theorems \ref{S(X)_Taimanov} and \ref{discrAminz1}.
	\item[(c)]We show that $Z(G/A)=S/A$ and $|S/A|=2$. The group $S/A$ has only one non-trivial element, that is, $S/A=\langle \pi(\tau)\rangle$, where $\pi:G\to G/A$ is the canonical projection and $\tau=(12)\in G$. Indeed, if $\sigma\in S$ and $\sigma\not\in A$, then $\tau\sigma\in A$ and so $\pi(\sigma)\in\langle \pi(\tau)\rangle$. Moreover $\tau\not \in A$. Since $S/A$ is a non-trivial normal subgroup of $G/A$ and it has size $2$, it is central; since $S/A$ is the unique non-trivial normal subgroup of $G/A$, $S/A=Z(G/A)$.
	\item[(d)]It follows from (c) that $G\not\in\mathfrak T$ by Proposition \ref{zg=1}(a).
	\item[(e)]By (d) $\zeta_A$ is not A-complete in view of Theorem \ref{almdiscr-solution}, hence $G\not\in\mathfrak A$.
	\item[(f)]Moreover it is possible to prove that $G/S\in\mathfrak T$. Consequently $G/S\in\mathfrak T_t$, being simple.
\end{itemize}
\end{exmp}

This is an example of a group $G$ which is not Arnautov but with $\delta_G$ A-complete. Moreover, since the subgroup of $G$ generated by the shift $\sigma$ is abelian and so not A-complete, while $\delta_G$ is A-complete, this example shows also that a subgroup of an A-complete group need not be A-complete.

\begin{exmp}
Consider the group $G:=SO_3(\mathbb R)$. As shown by Example \ref{SO_3(R)inT}, $G\in\mathfrak T$. Consequently $\delta_G$ is A-complete by Theorem \ref{discrAminz1}. Moreover the usual compact topology $\tau$ of $G$ is A-complete, because $\tau$ compact implies minimal, $Z(G)$ is trivial and so Proposition \ref{min,z01,a-min} applies.
\end{exmp}

A first step to find an answer to Question \ref{perfect-taimanov_question} is to consider the following.

\begin{question}\label{S(Z)/S}
\begin{itemize}
	\item[(a)] Does $S(\Z)/S_\omega(\Z)\in\mathfrak A$?
	\item[(b)] Does $SO_3(\mathbb R)\in\mathfrak A$?
\end{itemize}
\end{question}

\subsection{Totally Markov groups}\label{totally_Markov}

Our aim is to provide examples of groups in ${\mathfrak A}$. 

\medskip
The next results shows that for totally Markov groups the topologies are all almost trivial and so to verify if a continuous isomorphism of a totally Markov group is semitopological is simple, thanks to Corollary \ref{almdiscrsemitop}.

\begin{prop}\label{allalmdiscr}
A group $G\in\mathfrak M_t$ if and only if every group topology on $G$ is almost trivial.
\end{prop}
\begin{proof}
Suppose that $G\in\mathfrak M_t$ and let $\tau$ be a group topology on $G$. Then the quotient topology of $\tau$ on $G/N_\tau$ is Hausdorff and hence discrete, being $G\in\mathfrak M_t$. So $N_\tau$ is open in $(G,\tau)$ and therefore $\tau$ is almost trivial.

Suppose that the group $G\not \in\mathfrak M_t$. Then there exists a normal subgroup $N$ of $G$ such that there exists a Hausdorff non-discrete group topology $\sigma$ on $G/N$. Let $\pi:G\to G/N$ be the canonical projection and $\tau=\pi^{-1}(\sigma)$. Therefore $N_\tau=N$ (because $N=\bigcap\{V:V\in\mathcal V_{(G/N,\sigma)}(e_{G/N})\}$ in $G/N$). Since $\sigma$ is non-discrete $N$ is not open and so $\tau$ is not almost trivial.
\end{proof}

Proposition \ref{3space-at}, together with Proposition \ref{allalmdiscr}, immediately implies that $\mathfrak M_t$ is closed under extensions:

\begin{defn}
For a class of abstract groups ${\mathcal P}$ one says that ${\mathcal P}$ is {\em closed under extensions}, if a group $G$ belongs to ${\mathcal P}$ whenever $N\in {\mathcal P}$ and $G/N\in {\mathcal P}$ for some normal subgroup $N$ of $G$.
\end{defn}

Moreover we have the same result for $\mathfrak M$:

\begin{theo}\label{c.u.e.}
The classes $\mathfrak M$ and $\mathfrak M_t$ are closed under extensions. In particular, $\mathfrak M$ and $\mathfrak M_t$ are closed under finite direct products.
\end{theo}
\begin{proof}
That $\mathfrak M_t$ is closed under extensions is a direct consequence of Propositions \ref{3space-at} and \ref{allalmdiscr}.

Suppose that the group $G$ has a normal subgroup $N$ such that $N\in\mathfrak M$ and $G/N\in {\mathfrak M}$. We show that $G\in \mathfrak M$. To this end let $\tau$ be a Hausdorff group topology on $G$. Then $\tau\restriction_N=\delta_N$. Consequently: 
\begin{itemize}
	\item[(i)]$N$ is closed in $(G,\tau)$, and
	\item[(ii)]$\pi:(G,\tau)\to (G/N,\tau_q)$ is a local homeomorphism.
\end{itemize}
By (i) $(G/N,\tau_q)$ is Hausdorff and so discrete. In view of (ii) $\tau=\delta_G$.
\end{proof}

In view of Theorem \ref{discrAminz1}, a necessary condition for A-completeness of $\delta_G$ for a group $G$ is $Z(G)=\{e_G\}$. For Markov groups also the converse implication holds:

\begin{cor}\label{discrAmin}
Let $G\in\mathfrak M$. Then $G\in\mathfrak T$ if and only if $Z(G)=\{e_G\}$.
\end{cor}
\begin{proof}
If $G\in\mathfrak T$, apply Theorem \ref{discrAminz1}.

Suppose $Z(G)=\{e_G\}$. Then $T_G$ is Hausdorff by Lemma \ref{taimanov}(b) and so $T_G=\delta_G$.
\end{proof}

In the following proposition we characterize totally Markov groups which are A-complete or Arnautov. In particular it shows that for a totally Markov group it is equivalent to be Arnautov and to be totally Ta\u\i manov, which is precisely the answer to Question \ref{perfect-taimanov_question} in the particular case of totally Markov groups.

\begin{prop}\label{totM,A-complete,Arnautov}
Let $G\in\mathfrak M_t$.
\begin{itemize}
\item[(a)] If $\tau$ is a group topology on $G$, the following conditions are equivalent:
\begin{itemize}
	\item[(i)] $(G,\tau)$ is A-complete;
	\item[(ii)] $G/N_\tau\in\mathfrak T$;
	\item[(iii)] for every $N\triangleleft G$, if $[G,N]\leq N_\tau \leq N$, then $N=N_\tau$.
\end{itemize}
\item[(b)] The following conditions are equivalent:
\begin{itemize}
	\item[(i)] $G\in\mathfrak A$;
	\item[(ii)] $G\in\mathfrak T_t$;
	\item[(iii)] $Z(G/N)=\{e_{G/N}\}$ for every $N\triangleleft G$;
	\item[(iv)] $[G,N]=N$ for every $N\triangleleft G$.
\end{itemize}
\end{itemize}
\end{prop}
\begin{proof}
(a) The equivalence (i)$\Leftrightarrow$(ii) follows from Lemma \ref{allalmdiscr} and Theorem \ref{almdiscr-solution}. The equivalence (i)$\Leftrightarrow$(iii) follows from Lemma \ref{allalmdiscr} and Corollary \ref{almdiscrsemitop}.

\smallskip
(b) The equivalence (i)$\Leftrightarrow$(ii) follows from (a) and the equivalence (ii)$\Leftrightarrow$(iii) follows from Corollary \ref{discrAmin}.

(iii)$\Rightarrow$(iv) Let $N$ be a normal subgroup of $G$. Then $[G,N]$ is a normal subgroup of $G$ and $Z(G/[G,N])$ is trivial by hypothesis. Since $N/[G,N]\leq Z(G/[G,N])$ also $N/[G,N]$ is trivial, that is $N=[G,N]$.

(iv)$\Rightarrow$(i) By Lemma \ref{allalmdiscr} every group topology on $G$ is almost trivial. So let $L$ be a normal subgroup of $G$. For every normal subgroup $N$ of $G$ such that $[G,N]\leq L\leq N$, $N=L$ because $[G,N]=N$ by hypothesis. This proves that $\zeta_L$ is A-complete by (a). Consequently $G\in\mathfrak A$.
\end{proof}

This proposition covers Example \ref{simple-finite}.

\begin{cor}\label{finite-arnautov}
\begin{itemize}
	\item[(a)] A finite group $G\in\mathfrak A$ if and only if $G\in\mathfrak T_t$.
	\item[(b)] For every $G\in\mathfrak M$ simple, $G\in\mathfrak A$.
\end{itemize}
\end{cor}

In Example \ref{S(Z)} we have seen that $S(\Z)\not\in \mathfrak A$, but $S(\Z)/S_\omega(\Z)\in\mathfrak T_t$. In relation to Question \ref{S(Z)/S} we consider the following, which has also its own interest. In Example \ref{SO_3(R)inT} we have seen that $SO_3(\mathbb R)\in\mathfrak T_t$, but clearly $SO_3(\mathbb R)\not\in\mathfrak M$.

\begin{question}\label{S(Z)/S-markov}
Does $S(\Z)/S_\omega(\Z)\in \mathfrak M$?
\end{question}

A positive answer to this question would imply that $S(\Z)/S_\omega(\Z)\in \mathfrak A$, that is a positive answer to Question \ref{S(Z)/S}, in view of Corollary \ref{finite-arnautov}(b), since $S(\Z)/S_\omega(\Z)$ is simple. From another point of view, in order to answer Question \ref{S(Z)/S-markov}, it is possible to consider first Question \ref{S(Z)/S} which involves a weaker condition.

\begin{exmp}
Let $V=(\mathbb F_{p^m})^n$, where $m,n\in\mathbb N_+$, $p\in\mathbb P$ and $(n,p^m-1)=1$. Define $G$ to be the semidirect product of $SL(V)$ and $V$. Then $[G,V]=V$.  Moreover every normal subgroup of $G$ contains $V$ and so, since $SL(V)$ is simple, $V$ is the unique non trivial normal subgroup of $G$. Then $G\in\mathfrak A$ by Corollary \ref{finite-arnautov}(a).
\end{exmp}

\begin{exmp}\label{infinite-arnautov}
\begin{itemize}
	\item[(a)]Corollary \ref{finite-arnautov}(b) provides an example of an infinite Arnautov group. Indeed Shelah \cite{Sh} constructed a simple Markov (hence totally Markov) group $M$ under CH.
	\item[(b)]The group $M$ contains a subgroup isomorphic to $\mathbb Z$, which is abelian and so not in $\mathfrak A$.
	\item[(c)]In general a totally  Markov group need not be an Arnautov group, that is, $\mathfrak M_t\not\subseteq \mathfrak A$; for example $G:=M\times\mathbb Z(2)\in\mathfrak M_t$ but $G\not\in\mathfrak A$.
\end{itemize}
\end{exmp}

Item (b) of this example shows that ${\mathfrak A}$ is not stable under taking subgroups.

\begin{question}
Is ${\mathfrak A}$ stable under taking (finite) direct products? And under taking (finite) powers?
\end{question}

In the next example we give examples of Arnautov groups which are not simple. Moreover we see a particular case (that of Markov simple groups) in which finite powers of Arnautov groups are Arnautov.

\begin{exmp}
Let $M\in\mathfrak M$ be simple; by Corollary \ref{finite-arnautov}(b) $M\in\mathfrak A$. We show that $M^n\in \mathfrak M_t$ and  also $M^n\in\mathfrak A$, for every $n\in\mathbb N_+$.

\smallskip
Since $M\in\mathfrak M$ is simple, $M\in \mathfrak M_t$.
By Theorem \ref{c.u.e.} $M^n\in \mathfrak M_t$ for every $n\in\mathbb N_+$.
So $M^n\in\mathfrak A$ by Proposition \ref{totM,A-complete,Arnautov}(b): for every normal subgroup $N$ of $M^n$, $N=M^k$ for some $k\leq n$ up to topological isomorphisms, and consequently $[M^n,N]=[M^n,M^k]=M^k=N$.
\end{exmp}

The next are corollaries of Propositions \ref{3space-at} and \ref{allalmdiscr}.

\begin{cor}\label{allalmdiscr1<2}
Let $G$ be a group and $N_1\leq N_2$ be normal subgroups of $G$ with $N_2/N_1\in\mathfrak M_t$. Then every group topology $\tau$ on $G$ with $\zeta_{N_2}\leq \tau \leq \zeta_{N_1}$ is almost trivial.
In particular,
\begin{itemize}
	\item[(a)]if $N_2\in\mathfrak M_t$, then every group topology $\tau$ on $G$ with $\tau\geq\zeta_{N_2}$ is almost trivial; and
	\item[(b)]if $G/N_1\in\mathfrak M_t$, then every group topology $\tau$ on $G$ with $\tau\leq\zeta_{N_1}$ is almost trivial.
\end{itemize}
\end{cor}
\begin{proof}
(a) Since $N_2\in\mathfrak M_t$, by Proposition \ref{allalmdiscr} $\tau\restriction_{N_2}$ is almost trivial. Moreover $\tau_q\geq(\zeta_{N_2})_q=\delta_{G/{N_2}}$ on $G/N_2$, and so $\tau_q=\delta_{G/N_2}$ and in particular it is almost trivial. By Proposition \ref{3space-at} $\tau$ is almost trivial.

Obviously, $N_1\leq N_\tau\leq N_2$. Therefore, the quotient topology ${\tau}_q$ of $(G,\tau)$ with respect to $N_1$ satisfies $\delta_{G/N_1}\geq {\tau}_q \geq \zeta_{N_2/N_1}$. To the normal subgroup $N_2/N_1\in\mathfrak M_t$ of the group $G/N_1$ and ${\tau}_q \geq \zeta_{N_2/N_1}$ we apply (a) to claim that ${\tau}_q$ is almost trivial.
Since ${\tau}_q$ was obtained from $\tau$ via a quotient with respect to the $\tau$-indiscrete normal subgroup $N_1$, by Lemma \ref{almdiscrIndQ} $\tau$ is almost trivial.

(b) Follows from the previous part.
\end{proof}

\begin{cor}
Let $G$ be a group and $N_1\leq N_2$ be normal subgroups of $G$ with $[N_2:N_1]$ finite. Then $G$ admits only finitely many group topologies $\tau$ with $\zeta_{N_2}\leq \tau \leq \zeta_{N_1}$ and they are all almost trivial. 
\end{cor}
\begin{proof}
Apply Corollary \ref{allalmdiscr1<2} to conclude that every group topology $\tau$ on $G$ such that $\zeta_{N_2}\leq \tau \leq \zeta_{N_1}$ is almost trivial. Moreover these $\tau$ are finitely many because $[N_2:N_1]$ is finite.
\end{proof}

\begin{rem}
A group $G$ is \emph{hereditarily non-topologizable} in case every subgroup of $G$ is totally Markov \cite{Lu}. Thus
\begin{center}
hereditarily non-topologizable $\Rightarrow$ totally Markov $\Rightarrow$ Markov.
\end{center}
Consequently every group topology on a hereditarily non-topologizable group is almost trivial.

\smallskip
If a hereditarily non-topologizable group $G$ is Arnautov, then every quotient of $G$ is Arnautov.

\smallskip
While infinite Arnautov groups exist (see Example \ref{infinite-arnautov}(a)), it is not known if there exists any infinite non-topologizable group. The existence of such a group would solve an open problem from \cite{DU}.
\end{rem}

\section{Problem B}\label{pbB}

We start by underlying an important aspect of A$_\mathcal K$-completeness compared to $\mathcal K$-minimality, where $\mathcal K$ is a class of topological groups.
Indeed, let us recall first that A$_\mathcal G$-completeness coincides with A-completeness and implies A$_\mathcal K$-completeness (see Remark \ref{A_K-compl-rem}).
The $\mathcal K$-minimal groups are \emph{precisely} the indiscrete groups, whenever $\mathcal K$ contains all indiscrete groups. This fails to be true for A$_\mathcal K$-completeness. In fact, the group $G=S(\Z)$, equipped with either the discrete or the pointwise convergence topology, is $A$-complete (so A$_\mathcal K$-complete, for every $\mathcal K\subseteq \mathcal G$) as shown by Example \ref{S(Z)}(a,b). More generally for every non-trivial $G\in\mathfrak T$, the (obviously) non-indiscrete group $(G,\delta_G)$ is A-complete (so A$_\mathcal K$-complete, for every $\mathcal K\subseteq \mathcal G$) by Theorem \ref{discrAminz1}.

\medskip
As we have seen in Section \ref{pbA}  A-complete (i.e., A$_\mathcal G$-complete) groups are not easy to come by. In order to have a richer choice of groups, we consider A$_\mathcal K$-complete groups for appropriate subclasses $\mathcal K$ of $\mathcal G$.
In case the subclass $\mathcal K$ is completely determined by an algebraic property (i.e., for every group topology $\tau$ on $G$, $(G,\tau)\in\mathcal K$ if and only if $(G,\delta_G)\in\mathcal K$), then obviously a topological group $(G,\tau)\in \mathcal K$ is A$_\mathcal K$-complete if and only if it is A-complete.
A typical example to this effect is the class of all topological abelian groups, or more generally the class of all topological groups such that the underlying group belongs to a fixed variety $\mathcal V$ (in the sense of \cite{Neu}) of abstract groups.
We formulate an open question for a specific $\mathcal V$ in Question \ref{nilp-q}.

\medskip
In the sequel we consider subclasses $\mathcal K\subseteq\mathcal G$ of a different form, most often $\mathcal K\subseteq \mathcal H$.

\smallskip
Since $\mathcal H$-minimality coincides with minimality, A$_\mathcal H$-completeness is a generalization of minimality. It is a strict generalization in view of (a) of the next example.

\begin{exmp}\label{S(Z)-min}
\begin{itemize}
	\item[(a)]The group $(S(\Z),\delta_{S(\Z)})$ is A-complete, as shown by Example \ref{S(Z)}(b), and consequently A$_\mathcal H$-complete, but it is not minimal: $\delta_{S(\Z)}$ and the point-wise convergence topology $\mathcal T$ are both Hausdorff.
	\item[(b)] Let $G\in\mathfrak T$ be non-torsion. Then $(G,\delta_G)$ is A-complete by Theorem \ref{discrAminz1}, and in particular it is A$_\mathcal H$-complete. On the other hand, by our hypothesis there exists $x\in G$ of infinite order, that is $\langle x\rangle$ is abelian and so not A$_\mathcal H$-complete. This shows that in general a subgroup of an A$_\mathcal H$-complete group need not be A$_\mathcal H$-complete. (This is noted after Example \ref{S(Z)} for the particular case of $(S(\Z),\delta_{S(\Z)})$.)
\end{itemize}
\end{exmp}

Anyway A$_\mathcal{H}$-completeness coincides with minimality in the abelian case:

\begin{prop}	
If $G$ is an abelian group and $(G,\tau)\in\mathcal{H}$, then $(G,\tau)$ is A$_{\mathcal{H}}$-complete if and only if it is minimal.
\end{prop}

This proposition gives a partial answer to Problem B for the subclass of $\mathcal H$ of abelian topological groups. The problem remains open for the larger class $\mathcal H$:

\begin{question}\label{A_H-q}
When is a topological group $(G,\tau)\in\mathcal H$ A$_\mathcal H$-complete? And in which cases is $(G,\delta_G)$ A$_\mathcal H$-complete?
\end{question}

The next example, that extends Example \ref{S(Z)}(a), motivates Lemma \ref{zg=1->A-min=A_H-min}.

\begin{exmp}
For an infinite topologically simple (i.e., there exists no non-trivial closed normal subgroup) Hausdorff non-abelian group $(G,\tau)$, minimal implies A-complete. In fact $Z(G)=\{e_G\}$ and Proposition \ref{min,z01,a-min} applies.
\end{exmp}

The next lemma and corollary provide partial answers to Question \ref{A_H-q}. Lemma \ref{zg=1->A-min=A_H-min} in particular covers the previous example, since it implies that every minimal group with trivial center is A-complete (in view of the fact that minimal implies A$_\mathcal H$-minimal).

\begin{lem}\label{zg=1->A-min=A_H-min}
Let $G$ be a group with $Z(G)=\{e_G\}$ and let $\tau$ be a Hausdorff group topology on $G$. Then $(G,\tau)$ is A$_\mathcal H$-complete if and only if $(G,\tau)$ is A-complete.
\end{lem}
\begin{proof}
If $(G,\tau)$ is A-complete, then it is A$_\mathcal H$-complete.

Suppose that $(G,\tau)$ is A$_\mathcal H$-complete. Let $\sigma\leq\tau$ be a group topology on $G$ such that $id_G:(G,\tau)\to (G,\sigma)$ is semitopological. By Proposition \ref{t2,z=1,t2} $\sigma$ is Hausdorff. Then $\sigma=\tau$. This proves that $(G,\tau)$ is A-complete.
\end{proof}

This lemma implies Proposition \ref{min,z01,a-min}, since minimal groups are A$_\mathcal H$-complete.

\begin{cor}\label{A_H-min}
Let $G$ be a group. Then $Z(G)=\{e_G\}$ and $\delta_G$ is A$_\mathcal H$-complete if and only if $G\in\mathfrak T$.
\end{cor}
\begin{proof}
If $Z(G)=\{e_G\}$ and $\delta_G$ is A$_\mathcal H$-complete, then $\delta_G$ is A-complete by Lemma \ref{zg=1->A-min=A_H-min} and so $G\in\mathfrak T$ by Theorem \ref{discrAminz1}.

Assume that $G\in\mathfrak T$. By Theorem \ref{discrAminz1} $\delta_G$ is A-complete and so A$_\mathcal H$-complete. Moreover $Z(G)=\{e_G\}$ by Proposition \ref{zg=1}(a).
\end{proof}

Lemma \ref{zg=1->A-min=A_H-min} suggests the following question: is $Z(G)=\{e_G\}$ a \emph{necessary} condition for the validity of the implication $(G,\tau)$ A$_\mathcal H$-complete $\Rightarrow$ $(G,\tau)$ A-complete?
According to Corollary \ref{A_H-min} the answer is ``yes'' in case $\tau$ is the discrete topology.

\begin{prop}\label{SIN-A}
Let $(G,\tau)$ be a SIN Hausdorff group. If $(G,\tau)$ is A-complete, then $Z(G)=\{e_G\}$.
\end{prop}
\begin{proof}
Suppose that $Z(G)\neq\{e_G\}$.
We want to see that $(G,\tau)$ fails to be A-complete. Consider the topology $T:=\tau\wedge\zeta_{Z(G)}$, which has as a local base at $e_G$ the family $\mathcal B_T=\{U\cdot Z(G):U\in\mathcal V_{(G,\tau)}(e_G)\}$. Since $\tau$ is Hausdorff and $T$ is not Hausdorff (because $Z(G)\neq\{e_G\}$), $\tau>T$. So it remains to prove that $id_G:(G,\tau)\to (G,T)$ is semitopological. Since $(G,\tau)$ is SIN, it suffices to prove that for every $U\in\mathcal V_{(G,\tau)}(e_G)$ and for a fixed $g\in G$ there exists $V_g\in\mathcal B_T$ such that $[g,V_g]\subseteq U$ and then apply Proposition \ref{SINsemitop}. So let $U\in\mathcal V_{(G,\tau)}(e_G)$ and $g\in G$. Since $(G,\tau)$ is SIN, there exists $U'\in\mathcal V_{(G,\tau)}(e_G)$ such that $U'U'\subseteq U$ and $g U' g^{-1}\subseteq U'$. Let $V_g=U'\cdot Z(G)\in\mathcal B_T$. Then $[g,V_g]=[g,U']\subseteq U'U'\subseteq U$.
Since we have proved that $id_G:(G,\tau)\to (G,T)$ is semitopological and $\tau>T$, then $(G,\tau)$ fails to be A-complete.
\end{proof}

\begin{rem}\label{rem}
As a consequence of Lemma \ref{zg=1->A-min=A_H-min} and Proposition \ref{SIN-A} we have the following equivalence between A-completeness and the purely algebraic property of having trivial center.
Indeed, if $(G,\tau)\in\mathcal H$ is A$_\mathcal H$-complete, then $Z(G)=\{e_H\}$ implies $(G,\tau)$ A-complete by Lemma \ref{zg=1->A-min=A_H-min}. Moreover, if $(G,\tau)$ is SIN, in view of Proposition \ref{SIN-A} also the converse implication holds, that is, $(G,\tau)$ is A-complete if and only if $Z(G)=\{e_G\}$.
\end{rem}

\begin{cor}\label{SIN-A_H}
Let $(G,\tau)$ be a Hausdorff group with $Z(G)\neq\{e_G\}$.
\begin{itemize}
	\item[(a)]If $(G,\tau)$ is SIN and A$_\mathcal H$-complete, then it is not A-complete.
	\item[(b)]If $(G,\tau)$ is SIN and minimal, then it is not A-complete.
	\item[(c)]If $(G,\tau)$ is compact, then it is not A-complete.
\end{itemize}
\end{cor}

This corollary produces in particular examples of A$_\mathcal H$-complete groups which are not A-complete (e.g, compact groups with non-trivial center), showing that the implication $(G,\tau)$ A$_\mathcal H$-complete $\Rightarrow$ $(G,\tau)$ A-complete may fail to be true, also for non-discrete groups. In particular in Example \ref{exmp} shows a group, with non-trivial center, which does not admit any compact topology, but admits minimal linear (so SIN) topologies, that are not A-complete by Corollary \ref{SIN-A_H}.

\begin{prop}\label{A_H,non-A}
Let $G$ be a group such that $G\in\mathfrak T$ and let $F$ be a finite group. Then $\delta_{G\times F}$ is A$_\mathcal H$-complete.
\end{prop}
\begin{proof}
Let $\tau$ be a Hausdorff group topology on $G\times F$ and suppose that $id_{G\times F}:(G\times F,\delta_{G\times F})\to (G\times F,\tau)$ is semitopological. By Remark \ref{discrhom} $\tau\geq T_{G\times F}$. But $T_{G\times F}=T_G\times T_F=\delta_G\times T_F$ by Lemma \ref{tai-products}. So $\tau\geq \delta_G\times T_F$. Since $\tau$ is Hausdorff, $\tau=\delta_{G\times F}$, and this proves that $\delta_{G\times F}$ is A$_\mathcal H$-complete.
\end{proof}

Using this proposition we can give examples of A$_\mathcal H$-complete groups which are not A-complete, as the following. Another example of an A$_\mathcal H$-complete group which is not A-complete is in Example \ref{exmp}.

\begin{exmp}
Let $G=S(\Z)\times \Z(2)$. By Theorem \ref{S(X)_Taimanov}(a) ${S(\Z)}\in\mathfrak T$. Then $(G,\delta_G)$ is A$_\mathcal H$-complete by Proposition \ref{A_H,non-A}. Since $Z(G)=\{id_\Z\}\times \Z(2)$ is not trivial, $G\not\in\mathfrak T$ by Proposition \ref{zg=1}(a). Consequently $G$ is not A-complete by Theorem \ref{discrAminz1}.
\end{exmp}

\begin{exmp}\label{exmp}
Let $p\in\mathbb P$ and let $G$ be the group $\mathbb H_\Z$ (see Example \ref{H_Z}) equipped with the product topology $T=P(\tau_p,\tau_p,\tau_p)$ where $\tau_p$ is the $p$-adic topology of $\Z$. A base of $T$ is given by the family of the (normal) subgroups formed by the matrices
of the form $\begin{pmatrix}1 & p^n \Z & p^n \Z\cr
  0 &  1      &  p^n\Z \cr
0    &   0  & 1
\end{pmatrix}.$ Clearly $G$ is SIN.
Then $(G,T)$ is minimal \cite{D,DM}, so A$_\mathcal H$-complete.  Moreover $(G,T)$ is A-complete by Corollary \ref{SIN-A_H}.
\end{exmp}

Considering SIN groups in Example \ref{sin-example}, Proposition \ref{SIN-A} and Corollary \ref{SIN-A_H} we have weakened the commutativity from a topological point of view. A different way to weaken commutativity, but algebraically, is to consider nilpotent topological groups:

\begin{question}\label{nilp-q}
If $(G,\tau)$ is a nilpotent topological group, when is $(G,\tau)$ A-complete?
\end{question}

The following example is dedicated to a very particular case of this question. 

\begin{exmp}\label{K_R^0}
Consider the class $$\mathcal K_\mathbb R^0:=\{(\mathbb H_\mathbb R,P(\tau,\tau,\tau)): \tau\ \text{is a ring topology on}\ \mathbb R\},$$ where $P(\tau,\tau,\tau)$ denotes the product topology on $G$.
Then every $G\in\mathcal K_\mathbb R^0$ is A$_{\mathcal K_\mathbb R^0}$-complete.

\medskip
Indeed, let $\tau\geq\sigma$ be ring topologies on $\R$ such that $$(\mathbb H_\mathbb R,P(\tau,\tau,\tau)),(\mathbb H_\mathbb R,P(\sigma,\sigma,\sigma))\in\mathcal K_\mathbb R^0.$$ Suppose that $id_\mathbb R:(\mathbb H_\mathbb R,P(\tau,\tau,\tau))\to(\mathbb H_\mathbb R,P(\sigma,\sigma,\sigma))$ is semitopological. By Theorem \ref{semitop}, for every $U'=\begin{pmatrix}
1 & U & U \\
0 & 1 & U \\
0 & 0 & 1
\end{pmatrix}\in\mathcal V_{(\mathbb H_\mathbb R,P(\tau,\tau,\tau))}(e_{\mathbb H_\mathbb R})$ and $h=\begin{pmatrix}
1 & 1 & 0 \\
0 & 1 & 0 \\
0 & 0 & 1
\end{pmatrix}$ there exists $V_h=\begin{pmatrix}
1 & V & V \\
0 & 1 & V \\
0 & 0 & 1
\end{pmatrix}\in\mathcal V_{(\mathbb H_\mathbb R,P(\sigma,\sigma,\sigma))}(e_{\mathbb H_\mathbb R})$ such that $[h,V_h]\subseteq U'$. In particular this implies $V\subseteq U$ and hence $\sigma\geq \tau$, that is $\sigma=\tau$.
\end{exmp}

In a forthcoming paper \cite{DGB_x} we extend this result to the more general case of generalized Heisenberg groups on an arbitrary unitary ring $A$.

\section{Problem C}\label{pbC}

Problem C is about compositions of semitopological isomorphisms. In order to measure more precisely the level of being semitopological, we introduce the next notion.

\begin{defn}
Let $G$ be a group, $\sigma\leq\tau$ group topologies on $G$ and $n\in\mathbb N_+$. Then $id_G:(G,\tau)\to(G,\sigma)$ is \emph{$n$-step semitopological} if there exist $n-1$ group topologies $\sigma\leq\lambda_{n-1}\leq\ldots\leq\lambda_1\leq\tau$ on $G$ such that $id_G:(G,\tau)\to (G,\lambda_1),id_G:(G,\lambda_1)\to(G,\lambda_2),\ldots,id_G:(G,\lambda_{n-1})\to(G,\sigma)$ are semitopological.
\end{defn}

Obviously $id_G:(G,\tau)\to(G,\sigma)$ is $1$-step semitopological if and only if it is semitopological.
Moreover a continuous isomorphism of topological groups is composition of semitopological isomorphisms if and only if it is $n$-step semitopological for some $n\in\mathbb N_+$.

\medskip
Let $G$ be a non-trivial group. The \emph{lower central series} of $G$ is defined by $\gamma_1(G)=G$ and $\gamma_n(G)=[G,\gamma_{n-1}(G)]$ for every $n\in\mathbb N$, $n\geq 2$.
The \emph{upper central series} of $G$ is defined by $Z_0(G)=\{e_G\}$, $Z_1(G)=Z(G)$ and $Z_n(G)$ is such that $Z_n(G)/Z_{n-1}(G)=Z(G/Z_{n-1}(G))$ for every $n\in\mathbb N$, $n\geq 2$.
A group $G$ is nilpotent if and only if $\gamma_n(G)=\{ e_G \}$ for some $n\in\mathbb N_+$, if and only if $Z_m(G)=G$ for some $m\in\mathbb N_+$.
The minimum $n\in\mathbb N_+$ such that $\gamma_{n+1}(G)=\{ e_G \}$, that is, the minimum $n\in\mathbb N_+$ such that $Z_n(G)=G$, is the \emph{class of nilpotency} of $G$.

Our main theorem about $n$-step semitopological isomorphisms is the following. It is an answer to Problem C(a) in the particular case when the topologies on the domain and on the codomain are the discrete and the indiscrete one respectively.

\begin{theo}\label{n-step_semitop}
Let $G$ be a group and $n\in\mathbb N_+$. Then $id_G:(G,\delta_G)\to(G,\iota_G)$ is $n$-step semitopological if and only if $G$ is nilpotent of class $\leq n$.
\end{theo}
\begin{proof}
If $id_G:(G,\delta_G)\to(G,\iota_G)$ is $n$-step semitopological, then there exist $n-1$ group topologies $\lambda_{n-1}\leq\ldots\leq\lambda_1$ on $G$ such that $$id_G:(G,\delta_G)\to(G,\lambda_1),id_G:(G,\lambda_1)\to(G,\lambda_2),\ldots$$ $$\ldots,id_G:(G,\lambda_{n-2})\to(G,\lambda_{n-1}) ,id_G:(G,\lambda_{n-1})\to(G,\iota_G)$$ are semitopological. By Theorem \ref{ind}(b) $G'\subseteq V$ for every $V\in\mathcal V_{(G,\lambda_{n-1})}(e_G)$. Since $id_G:(G,\lambda_{n-2})\to(G,\lambda_{n-1})$ is semitopological, Theorem \ref{semitop} implies that for every $U\in\mathcal V_{(G,\lambda_{n-2})}(e_G)$ and for every $g\in G$ there exists $V_g\in\mathcal V_{(G,\lambda_{n-1})}(e_G)$ such that $[g,V_g]\subseteq U$. Consequently $[g,G']\subseteq U$ for every $U\in\mathcal V_{(G,\lambda_{n-2})}(e_G)$. Hence $\gamma_3(G)=[G,G']\subseteq U$ for every $U\in\mathcal V_{(G,\lambda_{n-2})}(e_G)$. Proceeding by induction we have that $\gamma_n(G)\subseteq U$ for every $U\in\mathcal V_{(G,\lambda_1)}(e_G)$. By Theorem \ref{discrete}(a) $c_G(g)$ is $\lambda_1$-open for every $g\in G$. Thus $\gamma_n(G)\subseteq Z(G)$ and this implies that $G$ is nilpotent of class $\leq n$ ($\gamma_{n+1}(G)=\{e_G\}$).

Conversely, if $G$ is nilpotent of class $\leq n$, consider on $G$ the group topologies $\zeta_{Z(G)},\zeta_{Z_2(G)},\ldots,\zeta_{Z_{n-1}(G)}$.
Then $id_G:(G,\delta_G)\to(G,\zeta_{Z(G)})$ is semitopological by Theorem \ref{discrete}(a) and $id_G:(G,\zeta_{Z_{n-1}(G)})\to(G,\iota_G)$ is semitopological because $G'\leq Z_{n-1}(G)$ since $G/Z_{n-1}(G)$ is abelian and applying Theorem \ref{ind}(b). For every $i=1,\ldots,n-1$, by Corollary \ref{almdiscrsemitop} $id_G:(G,\zeta_{Z_{i}(G)})\to(G,\zeta_{Z_{i+1}(G)})$ is semitopological if and only if $[G,Z_{i+1}(G)]\leq Z_{i}(G)$ and this holds true since $Z_{i+1}(G)/Z_{i}(G)=Z(G/Z_{i}(G))$.
\end{proof}

As a particular case of $n=2$ in this theorem, we find \cite[Example 12]{Ar}, which witnesses  that the composition of semitopological isomorphism is not semitopological in general. Indeed $id_G:(G,\delta_G)\to(G,\iota_G)$ is not semitopological, whenever $G$ is not abelian.

\medskip
For $n\in\mathbb N_+$, let $$n\text{-}\mathcal S:=\{f_n\circ\ldots\circ f_1:f_i\in\mathcal S\}.$$

Observe that $$\mathcal S=1\text{-}\mathcal S\subset 2\text{-}\mathcal S\subset\ldots\subset n\text{-}\mathcal S\subset {(n+1)}\text{-}\mathcal S\subset\ldots,$$ where all inclusions are proper by the previous theorem.

Define also $\infty\text{-}\mathcal S:=\bigcup_{n=1}^\infty n\text{-}\mathcal S$ and observe that it is closed under compositions.
Moreover $\infty\text{-}\mathcal S$ is closed also under taking subgroups, quotients and finite products, in the following sense:

\begin{lem}
Let $n\in\mathbb N_+$, let $G$ be a group and $\tau\geq\sigma$ group topologies on $G$ such that $id_G:(G,\tau)\to(G,\sigma)$ is $n$-step semitopological.
\begin{itemize}
	\item[(a)] If $A$ is a subgroup of $G$ then $id_G \restriction_A=id_A:A\to A$ is $n$-step semitopological.
	\item[(b)] If $A$ is a normal subgroup of $G$ then $id_{G/A}:(G/A,\tau_q)\to (G/A,\sigma_q)$ is $n$-step semitopological.
\end{itemize}
\end{lem}
\begin{proof}
(a) By hypothesis there exist $n-1$ group topologies $\sigma\leq\lambda_{n-1}\leq\ldots\leq\lambda_1\leq\tau$ on $G$ such that $$id_G:(G,\tau)\to(G,\lambda_1),id_G:(G,\lambda_1)\to(G,\lambda_2),\ldots,id_G:(G,\lambda_{n-1})\to(G,\sigma)$$ are semitopological. Theorem \ref{subgroupst}(a) implies that $$id_A:(A,\tau \restriction_A)\to(A,\lambda_1 \restriction_A),id_A:(A,\lambda_1 \restriction_A)\to(A,\lambda_2 \restriction_A),\ldots$$ $$\ldots,id_A:(A,\lambda_{n-1} \restriction_A)\to(A,\sigma \restriction_A)$$ are semitopological and so $id_A:(A,\tau \restriction_A)\to(A,\sigma \restriction_A)$ is $n$-step semitopological.

(b) Follows from Theorem \ref{quotientst}(b).
\end{proof}

The following lemma shows that for each $n\in\mathbb N_+$ the class $n\text{-}\mathcal S$ is closed under taking products. In particular it implies that $\infty\text{-}\mathcal S$ is closed under taking finite products.

\begin{lem}\label{prodstabnstep}
Let $n\in\mathbb N_+$, let $\{G_i:i\in I\}$ be a family of groups and $\{\tau_i:i\in I\}$, $\{\sigma_i:i\in I\}$ two families of group topologies such that $\sigma_i\leq\tau_i$ are group topologies on $G_i$ and $id_{G_i}:(G_i,\tau_i)\to(G_i,\sigma_i)$ is $n$-step semitopological for every $i\in I$. Then $\prod_{i\in I}id_{G_i}:\prod_{i\in I}(G_i,\tau_i)\to\prod_{i\in I}(G_i,\sigma_i)$ is $n$-step semitopological.
\end{lem}
\begin{proof}
It follows from Theorem \ref{productst}.
\end{proof}

The following example shows that $\infty\text{-}\mathcal S$ is not closed under taking infinite direct products and answers negatively (b) of Problem C. In fact we construct a continuous isomorphism which is not composition of semitopological isomorphisms.

\begin{exmp}\label{C(b)}
For every $n\in\mathbb N_+$ let $G_n$ be a nilpotent group of class $n$. Then $\prod_{n=1}^\infty id_{G_n}:\prod_{n=1}^\infty(G_n,\delta_{G_n})\to \prod_{n=1}^\infty(G_n,\iota_{G_n})$ is $n$-step semitopological for no $n\in\mathbb N_+$. Indeed $id_{G_{n+1}}:(G_{n+1},\delta_{G_{n+1}})\to(G_{n+1},\iota_{G_{n+1}})$ is not $n$-step semitopological whenever $n\in\mathbb N_+$, in view of Theorem \ref{n-step_semitop}, because $G_{n+1}$ is not nilpotent of class $\leq n$.
\end{exmp}

The next example is another particular case in which we answer Problem C(a).

\begin{exmp}
Let $n\in \mathbb N_+$, let $G$ be a totally Markov group and $\tau,\sigma$ group topologies on $G$. Every group topology on $G$ is almost trivial by Proposition \ref{allalmdiscr}. Then $id_G:(G,\tau)\to(G,\sigma)$ is $n$-step semitopological if and only if $\underbrace{[G,[G,[...[G}_n,N_\sigma]]]]\leq N_\tau$.

\medskip
In fact, suppose that $id_G:(G,\tau)\to(G,\sigma)$ is $n$-step semitopological. Then there exist group topologies $\sigma\leq\lambda_{n-1},\leq\ldots,\leq\lambda_{1}\leq\tau$ on $G$ such that $$id_G:(G,\tau)\to(G,\lambda_1),id_G:(G,\lambda_1)\to(G,\lambda_2),\ldots$$
$$\ldots,id_G:(G,\lambda_{n-1})\to(G,\sigma)$$ are semitopological. By Corollary \ref{almdiscrsemitop} $$[G,N_\sigma]\subseteq N_{\lambda_1},[G,N_{\lambda_1}]\subseteq N_{\lambda_2},\ldots,[G,N_{\lambda_{n-1}}]\subseteq N_\tau$$ and hence $\underbrace{[G,[G,[...[G}_n,N_\sigma]]]]\leq N_\tau$.

Assume that $\underbrace{[G,[G,[...[G}_n,N_\sigma]]]]\leq N_\tau$. Let $$N_{\lambda_1}=[G,N_\sigma], N_{\lambda_2}=[G,N_{\lambda_1}],\ldots,N_{\lambda_{n-1}}=[G,N_{\lambda_{n-2}}].$$ By Corollary \ref{almdiscrsemitop} and our assumption $id_G:(G,\tau)\to(G,\lambda_1),id_G:(G,\lambda_1)\to (G,\lambda_2),\ldots,id_G:(G,\lambda_{n-1})\to(G,\sigma)$ are semitopological.
\end{exmp}

%
%


\begin{thebibliography}{99}

\bibitem{Ar0} V.I. Arnautov, \emph{Semitopological isomorphisms of topological rings} (Russian), Mathematical Investigations (1969) 4:2 (12), 3--16.

\bibitem{Ar}  V.I. Arnautov, \emph{Semitopological isomorphisms of topological groups}, Bul. Acad. \c Stiin\c te Repub. Mold. Mat. (2004) no. 1, 15--25.

\bibitem{Ba}  S. Banach, {\it Ueber metrische Gruppen}, Studia Math. {\bf 3} (1931), 101--113.

\bibitem{Br} L. Brown, {\it Topologically complete groups}, Proc. Amer. Math. Soc. {\bf 35} (1972), 593--600. 

\bibitem{D} D. Dikranjan, {\em Recent advances in minimal topological groups}, Topology Appl. {\bf 85} (1998), no. 1--3, 53--91. 

\bibitem{DGB_x} D. Dikranjan, A. Giordano Bruno, \emph{Semitopological isomomorphisms for generalized Heisenberg groups}, work in progress.

\bibitem{DM} D. Dikranjan, M. Megrelishvili, \emph{Relative minimality and co-minimality of subgroups in topological groups}, submitted. 

\bibitem{DPS} D. Dikranjan, I. Prodanov, L. Stoyanov, \emph{Topological Groups: Characters,  Dualities  and  Minimal Group Topologies},  Pure and Applied Mathematics,  Vol. \textbf{130}, Marcel Dekker  Inc., New York-Basel, 1989.

\bibitem{DS} D. Dikranjan, D. Shakhmatov, {\em  Selected topics from the structure theory of topological groups}, in: E. Perl, Open Problems in Topology 2, Elsevier (2007), 389--406.

\bibitem{DU} D. Dikranjan, V. Uspenskij, {\em Categorically compact topological groups}, J. Pure Appl. Algebra {\bf 126} (1998), no. 1--3, 149--168. 

\bibitem{Do} D. Do\"\i tchinov,  {\it Produits de groupes topologiques minimaux}, Bull. Sci. Math. (2) {\bf 97} (1972), 59--64.

\bibitem{E} R. Engelking, \emph{General Topology}, Heldermann Verlag, Berlin, 1989.

\bibitem{Fuchs} L. Fuchs, \emph{Infinite abelian groups},  vol. I, Academic Press  New York and London, 1973.

\bibitem{AGB} A. Giordano Bruno, \emph{Semitopological homomorphisms}, to appear in Rend. Semin. Mat. Univ. Padova. 

\bibitem{Gr} D. L. Grant,  \emph{Topological groups which satisfy an open mapping theorem},  Pacific J. Math. {\bf 68} (1977), 411--423.

\bibitem{Hu} T. Husain, \emph{Introduction to topological groups}, Saunders, Philadelphia, 1966.

\bibitem{Kow} H. Kowalski, \emph{Beitrage sur topologischen albegra}, Math. Naschr. \textbf{11} (1954), 143--185.

\bibitem{Lu} G. Luk\'acs, \emph{Hereditarily non-topologizable groups}, arXiv:math/0603513v1 [math.GR].

\bibitem{M} M.  Megrelishvili, {\em Generalized Heisenberg groups and Shtern's question},  Georgian Math. J.  {\bf 11} (2004), no. 4, 775--782.

\bibitem{Neu} H. Neumann, \emph{Varieties of groups}, Springer-Verlag New York, Inc., New York, 1967, x+192 pp.

\bibitem{Ol} A. Yu. Ol$'$shanskii, \emph{A remark on a countable non-topologized group}, Vestnik Moskov Univ. Ser. I Mat. Mekh. (1980), 103 (in Russian).

\bibitem{Pt} V. Pt\' ak, {\it Completeness and the open mapping theorem}, Bull. Soc. Math. France {\bf 86} (1958), 41--74.

\bibitem{Rob} D. J. S. Robinson, \emph{A Course in the Theory of Groups}, Springer-Verlag, Berlin, 1982.

\bibitem{Sh} S. Shelah, \emph{On a problem of Kurosh, Jonsson groups and applications}, Word problems, II (Conf. on Decision Problems in Algebra, Oxford, 1976), pp. 373–-394, Stud. Logic Foundations Math., 95, North-Holland, Amsterdam-New York (1980).

\bibitem{Sl} M. Shlossberg, \emph{Minimality on Topological Groups and Heisenberg Type Groups}, submitted. 

\bibitem{St} R. M. Stephenson, Jr.,  {\it Minimal topological groups},  Math. Ann. {\bf 192} (1971), 193--195.

\bibitem{Su} L. Sulley,  {\it A note on B- and $B_r$-complete  topological
 Abelian  groups}, Proc. Cambr. Phil. Soc. {\bf 66} (1969), 275--279.
 
\bibitem{Tai} A. D. Ta\u\i manov, \emph{Topologizable groups. II.} (Russian) Sibirsk. Mat. Zh. \textbf{19} no. 5 (1978), 1201–-1203, 1216. (English translation: Siberian Math. J. \textbf{19} no. 5 (1978), 848--850 (1979).)

\bibitem{T} M. G. Tkachenko, \emph{Completeness of topological groups} (Russian), Sibirsk. Mat. Zh. \textbf{25} (1984), no. 1, 146--158.

\bibitem{T1} M. G. Tkachenko, \emph{Some properties of free topological groups} (Russian), Mat. Zametki \textbf{37} (1985), no. 1, 110--118, 139.


\end{thebibliography}
\end{document}